\documentclass[10pt,reqno,a4paper,twoside]{extarticle}

\usepackage[osf]{newtxtext}

\usepackage[left=1.8cm,right=1.8cm,top=2cm,bottom=2cm]{geometry}

\usepackage{graphicx, varioref, amsfonts}
\usepackage{amsthm,amssymb,amsmath,mathrsfs,mathtools,stmaryrd} 
\usepackage{dsfont}
\usepackage{upgreek} 
\usepackage{esint} 
\usepackage{breakurl}
\usepackage{empheq}
\usepackage{enumerate,enumitem}
\usepackage{hyperref}
\usepackage{cleveref}
\usepackage{soul}
\usepackage{cancel}
\usepackage{amscd}
\usepackage{srcltx}
\usepackage{ifthen}
\usepackage{float}
\usepackage{color}
\usepackage{comment}
\usepackage{chngcntr}
\usepackage{etoolbox}
\usepackage{fancyhdr}
\usepackage{breqn}
\usepackage{cite}
\usepackage[T1]{fontenc}
\usepackage{ulem}

\usepackage{titlefoot}

\usepackage{authblk} 
\theoremstyle{plain}
\newtheorem{lemma}{Lemma}[section]

\newtheorem{proposition}[lemma]{Proposition}

\newtheorem{theorem}[lemma]{Theorem}

\newtheorem{assumption}[lemma]{Assumption}

\theoremstyle{definition}

\newtheorem{definition}[lemma]{Definition}

\newtheorem{example}[lemma]{Example}
\newtheorem{remark}[lemma]{Remark}

\newlist{todolist}{itemize}{2}
\setlist[todolist]{label=$\square$}
\usepackage{pifont}

\numberwithin{equation}{section}
\begin{document}

\title{Stochastic optimal control in Hilbert spaces: $C^{1,1}$ regularity of the value function and optimal synthesis via viscosity solutions
}
\newcommand\shorttitle{Stochastic optimal control in Hilbert spaces: $C^{1,1}$ regularity of the value function and optimal synthesis}

\date{\hfill}

\author{Filippo de Feo\footnote{Dipartimento di Matematica, Politecnico di Milano, Milan, Italy, Department of Economics and Finance,  Luiss Guido Carli University, Rome, Italy, and Institut für Mathematik, Technische Universität Berlin, Berlin, Germany; Email: defeo@math.tu-berlin.de}}
\author{Andrzej {\'{S}}wi{\k{e}}ch\footnote{School of Mathematics, Georgia Institute of Technology, Atlanta, GA 30332, USA; Email: swiech@math.gatech.edu}}
\author{Lukas Wessels\footnote{School of Mathematics, Georgia Institute of Technology, Atlanta, GA 30332, USA; Email: wessels@gatech.edu}}
\newcommand\authors{Filippo de Feo, Andrzej {\'{S}}wi{\k{e}}ch, and Lukas Wessels}

\affil{}

\maketitle

\unmarkedfntext{\textit{Mathematics Subject Classification (2020) ---} 93E20, 49L25, 49L12, 49K45, 60H15, 49L20, 49N35, 35R15, 35K57, 34K50}


\unmarkedfntext{\textit{Keywords and phrases ---} optimal synthesis, viscosity solution, Hamilton--Jacobi--Bellman equation, dynamic programming, verification theorem, stochastic optimal control, infinite dimension, stochastic reaction-diffusion equations, stochastic delay equations}


\begin{abstract}
We study optimal control problems governed by abstract infinite dimensional stochastic differential equations using the dynamic programming approach. In the first part, we prove Lipschitz continuity, semiconcavity and semiconvexity of the value function under several sets of assumptions, and thus derive its $C^{1,1}$ regularity in the space variable. Based on this regularity result, we construct optimal feedback controls using the notion of $B$-continuous viscosity solutions for the associated Hamilton--Jacobi--Bellman equation. This is done in the case when the noise coefficient is independent of the control variable. We also discuss  applications of our results to optimal control problems governed by stochastic reaction-diffusion equations and, under economic motivations, stochastic delay differential equations.

\end{abstract}

\section{Introduction}
\noindent
\subsection{Optimal Control Problem}
In this paper, we show how to use the theory of viscosity solutions in Hilbert spaces to construct optimal feedback controls for control problems governed by abstract infinite dimensional stochastic differential equations (SDEs). Such evolution equations include, in particular, controlled semilinear stochastic partial differential equations (SPDEs) such as stochastic reaction-diffusion equations and stochastic wave equations as well as stochastic delay differential equations. We refer the reader to \cite[Section 2.6]{fabbri2017} for a comprehensive overview of such applications.

The general problem we study is the following.
Let $T <\infty$ be a fixed terminal time, let $H$ be a real, separable Hilbert space, let $\Lambda$ be a real, separable Banach space and let 
$\Lambda_0$ be a convex subset of $\Lambda$. For an initial time $0\leq t < T$, we consider the  optimal control problem of minimizing the cost functional
\begin{equation}\label{costfunctional}
	J(t,x;a(\cdot)) := \mathbb{E} \left [ \int_t^T l(X(s),a(s)) \mathrm{d}s + g(X(T)) \right ]
\end{equation}
over a class of appropriately defined admissible controls $a(\cdot):[t,T]\times\Omega \to \Lambda_0$, subject to the abstract SDE
\begin{equation}\label{state}
	\begin{cases}
		\mathrm{d}X(s) = [ A X(s) + b(X(s),a(s)) ] \mathrm{d}s + \sigma(X(s),a(s)) \mathrm{d}W(s),\quad s\in [t,T]\\
		X(t) = x\in H.
	\end{cases}
\end{equation}
Here, $l:H\times\Lambda_0\to\mathbb{R}, \Lambda_0\subset\Lambda$ and $g:H\to\mathbb{R}$ denote the running and terminal cost, respectively, $A:\mathcal{D}(A) \subset H \to H$ is a linear unbounded operator, $b:H\times \Lambda_0 \to H$ and $\sigma:H\times \Lambda_0 \to L_2(\Xi,H)$ denote the drift and the noise coefficient functions, respectively, and the SDE is driven by a cylindrical Wiener process $(W(s))_{s\in [t,T]}$ taking values in some real, separable Hilbert space $\Xi$ and defined on some filtered probability space $(\Omega, \mathscr{F},(\mathscr{F}_s^t)_{s\in[t,T]},\mathbb{P})$, where $(\mathscr{F}_s^t)_{s\in[t,T]}$ is the natural filtration generated by $(W(s))_{s\in [t,T]}$. The notation and the precise definitions will be explained in Section \ref{sec:defnotass}. Our goal in this paper is the construction of an optimal feedback control.

\subsection{Dynamic Programming Approach}

One of the major approaches to control problems of this kind is the dynamic programming approach introduced by Richard Bellman in the 1950s, see \cite{bellman1957}. In this approach, the central object of study is the so-called value function defined as
\begin{equation*}
	V(t,x) := \inf_{a(\cdot)} J(t,x;a(\cdot)),
\end{equation*}
which, provided it is sufficiently regular, can be used to derive sufficient optimality conditions (also known as verification theorems) as well as to construct optimal feedback controls (known as optimal synthesis). Obtaining the value function in the first place is a challenging task. To approach this problem, Bellman formally derived the, in general fully nonlinear, so-called Hamilton--Jacobi--Bellman (HJB) equation which should be satisfied by the value function. In our case the associated HJB equation has the form \cite{fabbri2017}
\begin{equation}\label{HJB}
\begin{cases}
	v_t + \langle Ax,Dv\rangle_H + \inf_{a\in \Lambda_0} \mathcal{F}(x,Dv,D^2v,a) =0,\quad (t,x)\in (0,T)\times H\\
	v(T,\cdot) = g,
\end{cases}
\end{equation}
where the Hamiltonian function $\mathcal{F}:H\times H \times S(H)\times \Lambda_0\to \mathbb{R}$ is given by
\begin{equation*}
	\mathcal{F}(x,p,P,a) := \frac12 \text{Tr} \left [ \sigma(x,a) \sigma^{\ast}(x,a) P \right ] + \langle b(x,a),p \rangle_H + l(x,a).
\end{equation*}
In a second step, assuming sufficient regularity, the solution of the HJB equation can also be used to derive verification theorems and perform optimal synthesis. For this reason, the construction of optimal feedback controls is intrinsically related to the study of the regularity of the solution of the HJB equation \eqref{HJB}. In many interesting cases, the value function is not sufficiently regular to satisfy the HJB equation in a classical sense. Therefore, various notions of solution have been applied to the study HJB equations.

\subsection{Results in Finite Dimension}
In finite dimensional spaces, existence, uniqueness and regularity of classical and strong solutions of parabolic HJB equations have been well studied (see, e.g., \cite{Dong, Lieber, Krylov1, Wang1, Wang2}) and classical verification theorems and existence of optimal feedback controls can be found for instance in \cite{FlRie, yong1999} (see also \cite{Krylov} for related results). When classical solutions do not exist, e.g., in the case of fully nonlinear and degenerate equations, the notion of viscosity solution has generally been used. We refer the reader for instance to \cite{bardi_1997,CIL, FlSon} for excellent overviews of the subject. In the absence of uniform parabolicity, regularity results tend to be short-time and rely on techniques which are typically based on comparison theorems or explicit representation formulas to show Lipschitz continuity, and semiconcavity and semiconvexity of viscosity solutions in the spatial variable. We refer the readers to \cite{CanSin, FlSon, Ishii, IshiiLions, Lions-book, yong1999} for various results in this direction. $C^{1,1}$ regularity results using a weaker notion of uniform ellipticity/parabolicity can be found in \cite{LionsIII} (see also \cite{Krylov} for earlier one- and two-sided second derivative estimates without the use of viscosity solutions). Applications of viscosity solutions to stochastic optimal control problems are discussed in classical monographs \cite{FlSon, yong1999}. A stochastic verification theorem and some results about optimal feedbacks using viscosity solutions can be found in \cite{yong1999,gozzi_swiech_zhou_2005,gozzi_swiech_zhou_2010}. Since the notion of viscosity solution does not require any regularity of the solution, these results rely on set-valued sub/superdifferentials (jets) of the value function.

\subsection{Results in Infinite Dimension}
In infinite dimensional Hilbert spaces, the problem is even more complicated since we have fewer techniques available and all HJB equations of type \eqref{HJB} are in some sense degenerate. Classical solutions are hard to come by and several different notions of generalized solutions, like viscosity solutions, mild solutions, mild solutions in $L^2$ spaces, solutions using backward SDEs, etc., have been introduced to resolve this issue. Early results can be found in \cite{barbu_daprato_1983}, see also \cite{DZ02} for the study of linear second order equations in Hilbert spaces. The recent monograph \cite{fabbri2017} contains a comprehensive overview of all the previously mentioned approaches. Apart from viscosity solutions, all these notions of generalized solutions assume some regularity of the solutions and their theories are mostly restricted to semilinear equations with a degree of non-degeneracy guaranteeing some smoothing properties of the transition semigroups generated by the linear parts of the equations, which in addition may also satisfy some structure conditions. Various regularity results using these notions of solutions are discussed in \cite[Chapters 4--6]{fabbri2017}. These regularity results can then be used to derive verification theorems and construct optimal feedback controls, see \cite[Sections 4.8, 5.5, 6.5 and the bibliographical notes of these sections]{fabbri2017}. In most cases, an approximation procedure by more regular solutions is used to circumvent the issue of lacking sufficient regularity. Another direction is to obtain weaker versions of It\^o's formula. This was done for instance in \cite{federico_gozzi_2018} and we refer to this paper for a discussion of other such results.

While the notion of viscosity solution adapts with some modifications to fully nonlinear degenerate equations in Hilbert spaces without unbounded operators, the presence of unbounded operators introduces serious difficulties. Firstly, there are various notions of viscosity solutions with subtle differences for such so-called ``unbounded'' equations. In this paper we use the approach of $B$-continuous viscosity solutions introduced in \cite{CLIV,CLV,SW94}; see  \cite[Chapter 3]{fabbri2017} for a comprehensive presentation. Secondly, since the lack of any a priori regularity persists from the finite dimensional case, verification theorems and optimal synthesis become even more complicated. For deterministic problems we refer for instance to \cite[Chapter 6]{LY}, \cite{can-fr, fgs_2008} for some results. We also mention the paper \cite{GomNur} which has results about the existence of minimizers, the value function and the associated HJB equation for deterministic calculus of variation problems in a Hilbert space. In the stochastic case, versions of the finite dimensional result from \cite{yong1999, gozzi_swiech_zhou_2005,gozzi_swiech_zhou_2010} appeared recently in \cite{chen2022,stannat_wessels_2021, wessels2023}, again using set-valued differentials of the value function. There are very few regularity results for viscosity solutions of HJB equations in Hilbert spaces. When the equation has only bounded and continuous terms, $C^{1,1}$ regularity in $x$ was proved for some classes of first order equations in \cite{bensoussan_2019, gangbo_meszaros_2020}. For special equations, regularity can also be derived from explicit representation formulas, see \cite{lasry1986}. For bounded second order equations, semiconcavity estimates were derived in \cite{lions-infdim1}. Moreover, when the HJB equation has some non-degeneracy on a closed subspace of $H$, partial $C^{1,1}$ regularity was derived using semiconcavity and classical finite dimensional non-degeneracy arguments. $C^{1,1}$ regularity in $x$ based on showing semiconcavity and semiconvexity of the value function for a class of semilinear degenerate parabolic second order HJB equations in a Hilbert space  related to a specific mean field control problems with common additive noise was obtained in \cite{mayorga_swiech_2022}. 
$C^{1,1}$ regularity of the value function for a control problem coming from mean field games was also proved in \cite{bensoussan_2020}. The equations considered in \cite{bensoussan_2020,mayorga_swiech_2022} do not have unbounded operators. In \cite{SwTe}, $W^{2,\infty}_Q$ regularity was proved for viscosity solutions of fully nonlinear obstacle problems with the so-called $Q$-elliptic operator, with the nuclear self-adjoint operator $Q>0$ governing the degenerate ellipticity of the PDE.

Regarding viscosity solutions of ``unbounded" HJB equations of type \eqref{HJB}, there are almost no results besides Lipschitz continuity. Semiconcavity estimates for the value function for deterministic control problems related to first order HJB equations are in \cite{LY}. Other available differentiability results were proven for problems with delays (rewritten in the Hilbert space $H=\mathbb R^n \times L^2([-d,0];\mathbb R^n)$): Partial finite dimensional $C^1$ and $C^{1,\alpha}$ regularity\footnote{This partial $C^1$ and $C^{1,\alpha}$ regularity is with respect to the so-called ``present'' variable $x_0 \in \mathbb R^n$.} for viscosity solutions of the HJB equations related to optimal control problems with delays only in the state were obtained in \cite{goldys_1} for first order equations, in \cite{rosestolato} for a linear second order Kolmogorov equation and recently in \cite{deFeo-Federico-Swiech} for a class of fully nonlinear second order equations. Regarding the case of problems with delays in the state and in the control, a directional $C^1$ regularity result was obtained in  \cite{tacconi} for a first order HJB equation (that is, for the deterministic case), while for stochastic problems we refer to  \cite{deFeo_2023} for remarks about why the method of \cite{deFeo-Federico-Swiech} could not be applied there to obtain the partial $C^{1,\alpha}$ regularity. Hence, up to our knowledge, there are no ``full'' $C^1$ regularity results available in the literature for viscosity solutions of (first or second order) HJB equations in Hilbert spaces with unbounded operators.


Using viscosity solutions, on the one hand we do not have any a priori regularity of a solution, however on the other hand, contrary to the case of mild solutions, we know in advance that the value function is the unique viscosity solution of \eqref{HJB}. This is a big help which was exploited in \cite{mayorga_swiech_2022} to perform optimal synthesis. In this work, based on the $C^{1,1}$ regularity of the value function in the space variable, the authors construct optimal feedback controls in a special case coming from a mean field control problem (mentioned above). Let us also mention that for deterministic problems with delays, existence of optimal feedback control was obtained in \cite{goldys_2} (using the partial regularity result of \cite{goldys_1}) and in \cite{tacconi}. A similar result was obtained in \cite{deFeoSwiech} for stochastic optimal control problems with delay using $B$-continuous viscosity solutions, a partial regularity result from \cite{deFeo-Federico-Swiech} and an elaborate double approximation procedure which permitted to find approximations of the value function which are regular enough so that  It\^o's formula could be used in order to prove a verification theorem.

\subsection{Our Results}
This manuscript has two main goals. The first is to prove $C^{1,1}$ regularity in the space variable $x$ of the value function for the stochastic optimal control problem \eqref{costfunctional}--\eqref{state}. The second is to show, in the semilinear case, how to use the HJB equation \eqref{HJB} and the notion of $B$-continuous viscosity solution to construct optimal feedback controls under this minimal regularity assumption.

In order to prove $C^{1,1}$ regularity in $x$ of the value function in the fully nonlinear case, we prove its semiconcavity and semiconvexity. The $C^{1,1}$ regularity is then well-known, see \cite{lasry1986}. The semiconcavity is an expected property and it is proved by adapting and modifying the proof in the finite dimensional case from \cite{yong1999}. The semiconvexity, on the other hand, is less standard and we prove it in three different cases: (1) when the running cost is strongly uniformly convex in the control variable; (2) when the state equation is linear and the cost is convex; (3) via comparison results for mild solutions of the state equation. In the third case, we need a comparison result for SPDEs, and existing results such as \cite{kotelenez1992,manthey1999,milian2002} do not apply to our situation. Therefore, along the way, we prove two comparison results, Theorems \ref{comparisonSPDEs1} and \ref{comparisonSPDEs}, which may be of independent interest.
 Our $C^{1,1}$ regularity results are, to the best of our knowledge, the first continuous Fr\'echet differentiability\footnote{With the approaches via mild solutions or BSDEs, usually only Gateaux differentiability is obtained.} results for the value function of optimal control problems in Hilbert spaces with unbounded operators (hence, under appropriate conditions, also for solutions of the corresponding HJB equation) when diffusion coefficients may be completely degenerate. We point out that our results and techniques can also be applied to deterministic problems, i.e., when $\sigma=0$. 

Regarding the second goal, in the semilinear case, we use the notion of $B$-continuous viscosity solutions of the HJB equation \eqref{HJB} to construct optimal feedback controls under the minimal regularity assumption $V(t,\cdot)\in C^{1,1}$. In this respect, under appropriate conditions, we extend  the method of \cite{mayorga_swiech_2022} to the case of general semilinear HJB equations in Hilbert spaces with unbounded operators. Our method is based on the simple observation that, if the value function $V$ is $C^{1,1}$ regular in $x$, then, under appropriate conditions, it is also a  viscosity solution of a linear Kolmogorov equation for which we have an explicit solution given by a Feynman--Kac formula. It then follows that the underlying SDE for this solution formula gives the optimal trajectory and the optimal feedback control. This way, a proof of a verification theorem (which requires to apply It\^o's formula to the value function) is avoided. The ideas behind this method are straightforward but we are not aware of any explicit use of it (other than \cite{mayorga_swiech_2022}) in the context of viscosity solutions. We remark that we obtain uniqueness of the solutions of the closed-loop equation in the original reference probability space, while usually only existence of weak  solutions is obtained via Girsanov's theorem. We also point out that the optimal synthesis result can be applied to deterministic problems. The results of the present paper are, up to our knowledge, the first results on optimal synthesis for control problems governed by SPDEs using viscosity solutions. Moreover, they apply to problems in general Hilbert spaces with unbounded operators, including the ones related to delay equations.

A limitation of using the notion of $B$-continuous viscosity solution is that the operator $A$ in \eqref{state} and \eqref{HJB} must be maximal dissipative. Moreover the theory requires $A$ to satisfy one of the two conditions, the so-called strong $B$-condition or the weak 
$B$-condition. Examples of operators satisfying these conditions can be found in \cite[Section 3.1.1]{fabbri2017}. Another example of an operator $A$ satisfying the weak $B$-condition is the operator coming from stochastic delay differential equations, see \cite{deFeo-Federico-Swiech}. However, it is worth pointing out that the $C^{1,1}$ regularity results obtained in Section \ref{sec:regularity} also hold when $A=0$, that is they apply to value functions of stochastic optimal control problems in Hilbert spaces driven by bounded evolution. Moreover, since the operator $A=0$ satisfies the strong $B$-condition with $B=I$ and $c_0=1$, the results and the optimal synthesis procedure presented in Section \ref{sec:optsynt} also hold when HJB equations do not have unbounded terms. Thus our results extend those of \cite{mayorga_swiech_2022} for optimal control problems with bounded evolution, where only a special class of such problems was considered.

\subsection{Outline of the Paper}
The paper is organized as follows. In Section \ref{sec:defnotass} we introduce the notation and the basic assumptions we will use throughout the rest of the paper. In Section \ref{sec:regularity} we prove regularity results for the value function and deduce its $C^{1,1}$ regularity. Based on this regularity result, in Section \ref{sec:optsynt} we construct optimal feedback controls. In Section \ref{weakBcase}, we extend our results to SDEs involving unbounded operators that satisfy the weak $B$-condition. Finally, in Section \ref{section:applications}, we discuss applications to controlled stochastic partial differential equations and, motivated by problems coming from economics, to controlled stochastic delay differential equations.

\section{Definitions, Notation and Basic Assumptions}\label{sec:defnotass}

Throughout the paper, $H$ is a real, separable Hilbert space with inner product $\langle\cdot,\cdot\rangle_H$ and norm $\|\cdot\|_H$, and 
$\Lambda_0$ is a convex subset of a real, separable Banach space $\Lambda$ with norm $\|\cdot\|_{\Lambda}$. The variables in $\Lambda$ will be denoted by $a$. Moreover, $\Xi$ is a real, separable Hilbert space and we fix $p>2$. We will write $B_R\subset H$ to denote the closed ball of radius $R>0$ in $H$.

We say that $\mu=(\Omega, \mathscr{F}, (\mathscr{F}^t_s)_{s\in [t,T]}, \mathbb{P}, (W(s))_{s\in [t,T]})$ is a reference probability space (cf. \cite[Definition 2.7]{fabbri2017}) if:
\begin{itemize}
    \item $(\Omega, \mathscr{F},\mathbb{P})$ is a complete probability space.
    \item $(W(s))_{s\in[t,T]}$ is a cylindrical Wiener process in $\Xi$ with $W(t) = 0$ $\mathbb{P}$--a.s.
    \item $\mathscr{F}^t_s = \sigma(\mathscr{F}^{t,0}_s,\mathcal{N})$, where $\mathscr{F}^{t,0}_s = \sigma(W(\tau) : t\leq \tau \leq s )$ is the filtration generated by $W$, and $\mathcal{N}$ is the collection of all $\mathbb{P}$--null sets in $\mathscr{F}$.
\end{itemize}
The set of admissible controls $\mathcal{U}_t$ consists of all $\mathscr{F}^t_s$--progressively measurable processes $a(\cdot):[t,T]\times\Omega \to \Lambda_0$ such that
\begin{equation}\label{eq:integr}
    \mathbb{E} \left [ \int_t^T \|a(s) \|_{\Lambda}^p \mathrm{d}s \right ] <\infty.
\end{equation}
\begin{remark}
    Note that condition \eqref{eq:integr} is only relevant if $\Lambda_0$ is unbounded. Furthermore, the requirement that $p>2$ is necessary in order to guarantee continuity of the paths of the solution of the state equation \eqref{state} which is used explicitly and implicitly in many parts of the paper.
\end{remark}
We define the value function $V:[0,T]\times H\to \mathbb{R}$ as
\begin{equation*}
	V(t,x) := \inf_{a(\cdot)\in \mathcal{U}_t} J(t,x;a(\cdot)),
\end{equation*}
where $J$ is given by equation \eqref{costfunctional}. Note that this definition of the value function coincides with the definition in which the infimum is also taken over all reference probability spaces. For more details, see \cite[Chapter 2 and in particular Theorem 2.22]{fabbri2017}.

Next, we introduce the spaces which will be used in the paper. Let $\mathcal{X}_1,\mathcal{X}_2$ be Banach spaces. By $L(\mathcal{X}_1,\mathcal{X}_2)$ we denote the space of bounded linear operators mapping from $\mathcal{X}_1$ to $\mathcal{X}_2$. For Hilbert spaces $H_1,H_2$, we denote by $L_2(H_1,H_2)$ the space of Hilbert--Schmidt operators. In the case $H=H_1=H_2$, we denote $L(H) = L(H,H)$, and $S(H)\subset L(H)$ denotes the subspace of self-adjoint operators. 
If $C\in L(H)$, $C\geq 0$ means that $\langle Cx,x\rangle_H \geq 0$ for all $x\in H$. We say that $B\in S(H)$ is strictly positive if $\langle Bx,x\rangle_H> 0$ for all $x\in H, x\not=0$.

We denote by $C^{1,2}((0,T)\times H)$ the space of functions $\varphi:(0,T)\times H\to\mathbb R$ such that, denoting the variables by $(t,x)$, $\partial_t\varphi, D\varphi, D^2\varphi$ are continuous, where $\partial_t\varphi$ is the partial derivative of $\varphi$ with respect to $t$ and $D\varphi, D^2\varphi$ are the first and second order Fr\'echet derivatives of $\varphi$ with respect to the $x$ variable. It $v:H\times\Lambda \to Z$, where $Z$ is a Banach space we will write $D_x v, D_a v$ to denote partial Fr\'echet derivatives of $v$ with respect to the variables $x$ and $a$ respectively. By $C^{1,1}(\mathcal{X}_1,\mathcal{X}_2)$ we denote the space of all once Fr\'echet differentiable mappings $v:\mathcal{X}_1\to \mathcal{X}_2$ with Lipschitz continuous derivative. In the case $\mathcal{X}_2=\mathbb{R}$, we write $C^{1,1}(\mathcal{X}_1) := C^{1,1}(\mathcal{X}_1,\mathbb{R})$. 

Throughout the paper, $A$ is a linear, densely defined, maximal dissipative operator in $H$. Note that $A$ is the generator of a $C_0$-semigroup of contractions on $H$. We will denote this semigroup by $(\mathrm{e}^{sA})_{s\geq 0}$. However, we refer to \cite[Remark 3.38]{fabbri2017} and \cite{rosestolato} for standard techniques that allow to treat the more general case in which $A-\lambda I$ is maximal dissipative for some $\lambda>0$. Since in some parts of the paper we will have to assume that the functions $b,\sigma$ are differentiable, to avoid any confusion we will assume that the functions $b,\sigma, l$ are defined on the whole space $H\times \Lambda$, however all conditions will be required to hold only on $H\times \Lambda_0$. In some parts of the paper we will also use the following two assumptions on the operator $A$.
\begin{assumption}\label{assumptionA}(Strong $B$-condition)
There is a strictly positive, self-adjoint operator $B\in L(H)$ such that $A^{\ast} B\in L(H)$ and for some $c_0\geq 0$
    \begin{equation*}
        -A^{\ast} B+c_0 B \geq I.
    \end{equation*}
\end{assumption}

\begin{assumption}\label{assumptionAw}(Weak $B$-condition)
There is a strictly positive, self-adjoint operator $B\in L(H)$ such that $A^{\ast} B\in L(H)$ and for some $c_0\geq 0$
    \begin{equation*}
        -A^{\ast} B+c_0 B \geq 0.
    \end{equation*}
\end{assumption}
We refer to \cite[Section 3.1.1]{fabbri2017} for the proof that the weak $B$-condition is always satisfied with $B=((I-A)(I-A^{\ast}))^{-\frac{1}{2}}$ and $c_0=1$. We also refer to \cite[Section 3.1.1]{fabbri2017} for examples of operators satisfying the strong and the weak $B$-conditions. An important case of an operator $A$ satisfying the weak $B$-condition is the operator coming from optimal control problems with delay from
\cite[Section 3.2]{deFeo-Federico-Swiech} (denoted $\tilde A$ there), where $B=(A^{-1})^{\ast} A^{-1}$ and $c_0=0$ are taken.

Using the operator $B$ from Assumption \ref{assumptionA} or Assumption \ref{assumptionAw}, we extend the space $H$ in the following way.
\begin{definition}
    Let $B\in S(H)$ be strictly positive. The space $H_{-1}$ is the completion of $H$ with respect to the norm
    \begin{equation*}
        \|x\|_{-1}^2 := \langle Bx,x \rangle_H.
    \end{equation*}
\end{definition}
Note that $H_{-1}$ is a Hilbert space when endowed with the inner product
\begin{equation*}
    \langle x,y\rangle_{-1} := \langle B^{\frac12} x, B^{\frac12} y \rangle_H.
\end{equation*}

\begin{remark}\label{rem:assumptions_weakstrong_B}
We remark that Assumptions \ref{assumptionA}, \ref{assumptionAw} are not used in order to obtain the regularity results in $H$ in Section \ref{sec:regularity}. Assumption  \ref{assumptionA} is needed for comparison for the associated HJB equation, including the continuity of $V(t,\cdot)$ with respect to the $\|\cdot\|_{-1}$ norm. Hence it is only used to obtain the results of Section \ref{sec:optsynt}. Assumption \ref{assumptionAw} is used for similar purposes and to obtain better regularity results for $V(t,\cdot)$ in the space $H_{-1}$, however, since the whole analysis of Section \ref{sec:optsyntw} is done in $H_{-1}$, it is needed for all results of Section \ref{weakBcase}.
\end{remark}

The notions of semiconcavity and semiconvexity will play a crucial role in our work.

\begin{definition}\label{def:semiconcave}
    A function $v:H\to\mathbb{R}$ is called semiconcave if there is a constant $C>0$ such that 
    \begin{equation*}
        \lambda v(x) + (1-\lambda) v(x^{\prime}) - v(\lambda x + (1-\lambda)x^{\prime}) \leq C \lambda(1-\lambda) \| x - x^{\prime} \|_H^2
    \end{equation*}
    for all $\lambda \in [0,1]$ and $x,x^{\prime}\in H$. $v$ is called semiconvex if $-v$ is semiconcave.
\end{definition}
The constant $C$ in Definition \ref{def:semiconcave} is called the {\it semiconcavity constant} of $v$. The semiconcavity constant of $-v$ is called the {\it semiconvexity constant} of $v$.

We also introduce the following notion of semiconcavity and semiconvexity in $H_{-1}$.
\begin{definition}\label{def:semiconvex}
    A function $v:H\to\mathbb{R}$ is called semiconcave in $H_{-1}$ if there is a constant $C>0$ such that 
    \begin{equation*}
        \lambda v(x) + (1-\lambda) v(x^{\prime}) - v(\lambda x + (1-\lambda)x^{\prime}) \leq C \lambda(1-\lambda) \| x - x^{\prime} \|_{-1}^2
    \end{equation*}
    for all $\lambda \in [0,1]$ and $x,x^{\prime}\in H$. $v$ is called semiconvex in $H_{-1}$ if $-v$ is semiconcave.
\end{definition}
When it is clear that we deal with functions $v$ which are semiconcave/semiconvex in $H_{-1}$, the constant $C$ in Definition \ref{def:semiconvex} will also be called the semiconcavity constant of $v$ and the semiconcavity constant of $-v$ will be called the semiconvexity constant of $v$. Note that the function $v$ in Definition \ref{def:semiconvex} does not need to be defined on $H_{-1}$ but only on $H$.

\section{Regularity of the Value Function}\label{sec:regularity}

\subsection{Lipschitz Continuity}\label{lipschitz}
In this section, we prove Lipschitz continuity of the value function. We impose the following assumptions on the coefficients of the state equation.

\begin{assumption}\label{bsigmalipschitzfirstvariable}
\begin{enumerate}[label=(\roman*)]
    \item The function $b$ is continuous on $H\times \Lambda_0$ and there exists a constant $C>0$ such that
    \begin{equation*}
        \|b(x,a)-b(x^{\prime},a) \|_H \leq C \|x-x^{\prime} \|_H
    \end{equation*}
    for all $x,x^{\prime}\in H$ and $a\in\Lambda_0$.
    \item There exists a constant $C>0$ such that
    \begin{equation*}
        \|b(x,a)\|_H \leq C(1+\|x\|_H+\|a\|_{\Lambda} )
    \end{equation*}
    for all $x\in H$ and $a\in\Lambda_0$.
    \item The function $\sigma$ is continuous on $H\times \Lambda_0$ and there exists a constant $C>0$ such that
    \begin{equation*}
        \|\sigma(x,a)-\sigma(x^{\prime},a) \|_{L_2(\Xi,H)} \leq C \|x-x^{\prime} \|_H
    \end{equation*}
    for all $x,x^{\prime}\in H$ and $a\in\Lambda_0$.
    \item There exists a constant $C>0$ such that
    \begin{equation*}
        \|\sigma(x,a)\|_{L_2(\Xi,H)} \leq C(1+\|x\|_H+\|a\|_{\Lambda} )
    \end{equation*}
    for all $x\in H$ and $a\in\Lambda_0$.
\end{enumerate}
\end{assumption}

\begin{assumption}\label{bsigmalipschitz}
	\begin{enumerate}[label=(\roman*)]
		\item There exists a constant $C>0$ such that
		\begin{equation*}
			\| b(x,a) - b(x^{\prime},a^{\prime} ) \|_H \leq C (\|x-x^{\prime}\|_H + \|a-a^{\prime}\|_{\Lambda} )
		\end{equation*}
		for all $x,x^{\prime}\in H$ and $a,a^{\prime}\in \Lambda_0$.
		\item There exists a constant $C>0$ such that
		\begin{equation*}
			\| \sigma(x,a) - \sigma(x^{\prime},a^{\prime} ) \|_{L_2(\Xi,H)} \leq C (\|x-x^{\prime}\|_H + \|a-a^{\prime}\|_{\Lambda} )
		\end{equation*}
		for all $x,x^{\prime}\in H$ and $a,a^{\prime}\in \Lambda_0$.
	\end{enumerate}
\end{assumption}
We will discuss the case of time-dependent coefficients in Remark \ref{rem:time_dependent_coeff}.

Under these assumptions, for fixed $t \in [0,T], x \in H ,$ $ a(\cdot) \in \mathcal U_t$, the state equation \eqref{state} has a unique solution with continuous trajectories, denoted by $X(s;t,x,a(\cdot))$, see, e.g., \cite[Chapter 6, Theorem 6.5]{chow2007}; see also \cite[Theorem 7.2]{DZ14}\footnote{Note that, in contrast to our situation, in \cite{DZ14}, the coefficients of the state equation, depending on 
$\omega$, are uniformly bounded in $\omega$ (while in our case this is not true due the their unboundedness in the control). However, due to  \eqref{eq:integr}, a straightforward modification of the arguments leads to the desired existence and uniqueness result.}. The solution satisfies
\begin{equation}\label{xeq:mildsupest}
	\mathbb{E}\left [\sup_{s\in[t,T]}\|X(s;t,x,a(\cdot))\|_H^p\right ]\leq C_p(T)\left(1+\|x\|_H^p
	+\mathbb{E}\int_t^T \|a(r) \|_{\Lambda}^p \mathrm{d}r \right).
\end{equation}

For $x_0,x_1\in H$ and $a_0(\cdot),a_1(\cdot)\in \mathcal{U}_t$, define
\begin{equation}\label{x1x0differentcontrols}
	\begin{cases}
		X_0(s) = X(s;t,x_0,a_0(\cdot))\\
		X_1(s)=X(s;t,x_1,a_1(\cdot)). 
	\end{cases}
\end{equation}

In the proof of Lemma \ref{estimatex1x0one} below and later, we will use \cite[Proposition 1.166]{fabbri2017}. Note that in \cite{fabbri2017}, the coefficients of the state equation are uniformly bounded in the control variable, while in our case they are not. Nevertheless, adapting the proof to our case, the functions $f$ and $\Phi$ defined in the proof of \cite[Proposition 1.166]{fabbri2017} are still in $M^p_{\mu}(t,T;H)$ and $\mathcal{N}_I^p(t,T;H)$ (since $Q=I$ in our case), respectively, and thus the proof can be carried out verbatim.

We have the following estimates.
\begin{lemma}\label{estimatex1x0one}
	\begin{enumerate}[label=(\roman*)]
		\item Let Assumption \ref{bsigmalipschitzfirstvariable} be satisfied. Let $a_0(\cdot) = a_1(\cdot) = a(\cdot)\in \mathcal{U}_t$. Then, there are constants $C_1,C_2\geq 0$ independent of $T$, such that
		\begin{equation*}
			\mathbb{E} \left [ \sup_{s\in [t,T]} \| X_1(s) - X_0(s) \|_H^{2} \right ] \leq C_1 \mathrm{e}^{C_2(T-t)}\|x_1 - x_0 \|_H^{2}
		\end{equation*}
        for all $x_0,x_1\in H$ and $a(\cdot)\in \mathcal{U}_t$.
		\item Let Assumption \ref{bsigmalipschitz} be satisfied. Then, there are constants $C_1,C_2\geq 0$ independent of $T$, such that
		\begin{equation}\label{eq:strongest}
			\mathbb{E} \left [ \sup_{s\in [t,T]} \| X_1(s) - X_0(s) \|_H^{2} \right ] \leq C_1 \mathrm{e}^{C_2(T-t)} \left ( \|x_1 - x_0 \|_H^{2} + \mathbb{E} \left [ \int_t^T \| a_1(s) - a_0(s) \|_{\Lambda}^2 \mathrm{d}s \right ] \right )
		\end{equation}
        for all $x_0,x_1\in H$ and $a_0(\cdot), a_1(\cdot) \in \mathcal{U}_t$.
	\end{enumerate}
\end{lemma}

\begin{proof}
    We are only going to prove part $(ii)$; the proof of part $(i)$ follows along the same lines.

	Applying \cite[Proposition 1.166]{fabbri2017}, we have
	\begin{equation}\label{itosformula3}
		\begin{split}
			&\| X_1(s) - X_0(s) \|_H^2\\
			&\leq \| x_1 - x_0 \|_H^2 + 2 \int_t^s \langle b(X_1(r),a_1(r)) - b(X_0(r),a_0(r)) , X_1(r) - X_0(r) \rangle_H \mathrm{d}r\\
			&\quad + \int_t^s \| \sigma(X_1(r),a_1(r)) - \sigma(X_0(r),a_0(r)) \|_{L_2(\Xi,H)}^2 \mathrm{d}r\\
			&\quad + 2 \int_t^s \langle X_1(r)-X_0(r), (\sigma(X_1(r),a_1(r)) - \sigma(X_0(r),a_0(r))) \mathrm{d}W(r) \rangle_H.
		\end{split}
	\end{equation}
	Due to Burkholder--Davis--Gundy inequality (see, e.g., \cite[Theorem 1.80]{fabbri2017}) and the Lipschitz continuity of $\sigma$, we have for any $t<T^{\prime}\leq T$,
	\begin{equation*}
		\begin{split}
			&\mathbb{E} \left [ \sup_{s\in [t,T^{\prime}]} \left | \int_t^s \langle X_1(r)-X_0(r), (\sigma(X_1(r),a_1(r)) - \sigma(X_0(r),a_0(r))) \mathrm{d}W(r) \rangle_H \right | \right ]\\
			&\leq C \mathbb{E} \left [ \left ( \int_t^{T^{\prime}} \| X_1(s) - X_0(s) \|_H^2 \left ( \| X_1(s) - X_0(s) \|_H^2 + \| a_1(s) - a_0(s) \|_{\Lambda}^2 \right ) \mathrm{d}s \right )^{\frac12} \right ]\\
			&\leq \frac14 \mathbb{E} \left [ \sup_{s\in [t,T^{\prime}]} \| X_1(s) - X_0(s) \|_H^2 \right ] + C \mathbb{E} \left [ \int_t^{T^{\prime}} \| X_1(s) - X_0(s) \|_H^2 + \| a_1(s) - a_0(s) \|_{\Lambda}^2 \mathrm{d}s \right ].
		\end{split}
	\end{equation*}
	Therefore, taking the supremum over $[t,T^{\prime}]$ in \eqref{itosformula3}, taking the expectation, and using the Lipschitz continuity of $b$ and $\sigma$, we obtain
	\begin{equation*}
		\begin{split}
			&\mathbb{E} \left [ \sup_{s\in [t,T^{\prime}]} \| X_1(s) - X_0(s) \|_H^2 \right ]\\
			&\leq C \| x_1 - x_0 \|_H^2 + C \int_t^{T^{\prime}} \mathbb{E} \left [ \sup_{r\in [t,s]} \| X_1(r) - X_0(r) \|_H^2 + \| a_1(s) - a_0(s) \|_{\Lambda}^2 \right ] \mathrm{d}s.
		\end{split}
	\end{equation*}
	Applying Gr\"onwall's inequality yields the claim.
\end{proof}

\begin{assumption}\label{lglipschitzfirstvariable}
\begin{enumerate}[label=(\roman*)]
    \item The function $l$ is continuous on $H\times\Lambda_0$ and there exists a constant $C>0$ such that
    \begin{equation*}
      |l(x,a) - l(x^{\prime},a)| \leq C \|x-x^{\prime}\|_H
    \end{equation*}
    for all $x,x^{\prime}\in H$ and $a\in \Lambda_0$.
     \item There exists a constant $C>0$ such that
    \begin{equation*}
      |l(x,a)| \leq C (1+\|x\|_H+\|a\|_{\Lambda}^2)
    \end{equation*}
    for all $x\in H$ and $a\in \Lambda_0$.
    \item There exists a constant $C>0$ such that
    \begin{equation*}
        |g(x) - g(x^{\prime})| \leq C \|x-x^{\prime}\|_H
    \end{equation*}
    for all $x,x^{\prime}\in H$.
    \item If $\Lambda_0$ is unbounded we assume that $l$ and $g$ are bounded from below.
\end{enumerate}
\end{assumption}

\begin{remark}\label{rem:time_dependent_coeff}
    All the results in the present paper extend to the case that $b$, $\sigma$, $l$ also depend on time $t$. In this case, we also require that for every $R>0,$ the functions $b(\cdot,x,a)$, $\sigma(\cdot,x,a)$, $l(\cdot,x,a)$ are uniformly continuous, uniformly in $x \in H,$ $\|x\|_H \leq R,$  $a\in \Lambda_0$.
\end{remark}

Assumption {\it{(iv)}} above is made to guarantee that the value function is well defined. It can be replaced by a different assumption guaranteeing an appropriate growth of $l$ in the control variable or by imposing restrictions on the quantity \eqref{eq:integr} for admissible controls, however we do not pursue this direction here. We note that it follows easily from \eqref{xeq:mildsupest} and Assumption \ref{lglipschitzfirstvariable} that $|V(t,x)|\leq C(1+\|x\|_H)$ for all $(t,x)\in [0,T]\times H$.
\begin{theorem}\label{th:V_lip}
    Let Assumptions \ref{bsigmalipschitzfirstvariable} and \ref{lglipschitzfirstvariable} be satisfied. Then there are constants $C_1,C_2$ independent of $T$, such that
	\begin{equation*}
		|V(t,x) - V(t,y)| \leq C_1 \mathrm{e}^{C_2(T-t)} \|x-y\|_H
	\end{equation*}
	for all $t\in [0,T]$ and $x,y\in H$.
\end{theorem}

\begin{proof}
    Fix $t\in [0,T]$, $x_0,x_1\in H$ and $a(\cdot)\in\mathcal{U}_t$. Let $X_0$ and $X_1$ be given by equation \eqref{x1x0differentcontrols} with $a_0(\cdot) = a_1(\cdot) = a(\cdot)$. Using Lipschitz continuity of $l$ and $g$ in the $x$-variable and Lemma \ref{estimatex1x0one}(i) we have
    \begin{equation*}
    \begin{split}
        & |J(t,x_1;a(\cdot)) - J(t,x_0;a(\cdot))| \\
        &\leq C \mathbb{E} \left [ \int_t^T \| X_1(s) - X_0(s) \|_H \mathrm{d}s + \| X_1(T) - X_0(T) \|_H \right ] \\
        &\leq C (T-t+1) \mathbb{E} \left [ \sup_{s\in [t,T]} \| X_1(s) - X_0(s) \|_H \right ] 
        \leq C_1 \mathrm{e}^{C_2(T-t)} \| x_1-x_0 \|_H
    \end{split}
    \end{equation*}
   for some $C_1,C_2\geq 0$.
   \end{proof}
We point out that if in addition Assumption \ref{assumptionA} is satisfied, using \cite[Lemma 3.23]{fabbri2017} we also have that $V(t,\cdot)$ is Lipschitz continuous with respect to the $\|\cdot\|_{-1}$ norm for every $0\leq t<T$, however the Lipschitz constant blows up as $t$ approaches $T$.
\subsection{Semiconcavity}\label{semiconcavity}

\begin{assumption}\label{bsigmafrechetfirstvariable}
\begin{enumerate}[label=(\roman*)]
    \item Let $b:H\times\Lambda \to H$ be Fr\'echet differentiable in the first variable and let there be a constant $C>0$ such that
    \begin{equation*}
        \| D_xb(x,a) - D_xb(x^{\prime},a) \|_{L(H)} \leq C \|x-x^{\prime}\|_H
    \end{equation*}
    for all $x,x^{\prime}\in H$ and $a\in \Lambda_0$.
    \item Let $\sigma:H\times\Lambda \to L_2(\Xi,H)$ be Fr\'echet differentiable in the first variable and let there be a constant $C>0$ such that
    \begin{equation*}
        \| D_x\sigma(x,a) - D_x\sigma(x^{\prime},a) \|_{L(H,L_2(\Xi,H))} \leq C \|x-x^{\prime}\|_H
    \end{equation*}
    for all $x,x^{\prime}\in H$ and $a\in \Lambda_0$.
\end{enumerate}
\end{assumption}

\begin{assumption}\label{bsigmafrechet}
	\begin{enumerate}[label=(\roman*)]
		\item Let $b:H\times\Lambda\to H$ be Fr\'echet differentiable and let there be a constant $C>0$ such that
		\begin{equation*}
			\| D_xb(x,a) - D_xb(x^{\prime},a^{\prime})\|_{L(H)} + \| D_ab(x,a) - D_ab(x^{\prime},a^{\prime}) \|_{L(\Lambda,H)} \leq C (\|x-x^{\prime}\|_H + \|a-a^{\prime}\|_{\Lambda})
		\end{equation*}
		for all $x,x^{\prime}\in H$ and $a,a^{\prime}\in \Lambda_0$.
		\item Let $\sigma:H\to L_2(\Xi,H)$ be Fr\'echet differentiable and let there be a constant $C>0$ such that
		\begin{equation*}
			\| D\sigma(x) - D\sigma(x^{\prime}) \|_{L(H,L_2(\Xi,H))} \leq C \|x-x^{\prime}\|_H
		\end{equation*}
		for all $x,x^{\prime}\in H$.
	\end{enumerate}
\end{assumption}

Recall the definition of $X_0$ and $X_1$ given in equation \eqref{x1x0differentcontrols}. Furthermore, for $\lambda\in [0,1]$ we introduce
\begin{equation}\label{lambdadefinition}
	\begin{cases}
		a_{\lambda}(s) = \lambda a_1(s) + (1-\lambda) a_0(s)\\
		x_{\lambda} = \lambda x_1 +(1-\lambda)x_0\\
		X_{\lambda}(s) = X(s;t,x_{\lambda},a_{\lambda}(\cdot))\\
		X^{\lambda}(s) = \lambda X_1(s) +(1-\lambda)X_0(s).  
	\end{cases}
\end{equation}

\begin{lemma}\label{estimatexlambdasamecontrol}
\begin{enumerate}[label=(\roman*)]
	\item Let Assumptions \ref{bsigmalipschitzfirstvariable} and \ref{bsigmafrechetfirstvariable} be satisfied. Let $a_0(\cdot) = a_1(\cdot) = a(\cdot)\in \mathcal{U}_t$. There are constants $C_1,C_2\geq 0$ independent of $T$, such that
	\begin{equation*}
		\mathbb{E} \left [ \sup_{s\in [t,T]} \| X^{\lambda}(s) - X_{\lambda}(s) \|_H \right ] \leq C_1 \mathrm{e}^{C_2(T-t)} \lambda (1-\lambda) \|x_1-x_0\|_H^2
	\end{equation*}
    for all $\lambda \in [0,1]$, $x_0,x_1\in H$ and $a(\cdot)\in \mathcal{U}_t$.
	\item Let Assumptions \ref{bsigmalipschitz} and \ref{bsigmafrechet} be satisfied. There are constants $C_1,C_2\geq 0$, such that
	\begin{equation}\label{eq:xlambda}
			\mathbb{E} \left [ \sup_{s\in [t,T]} \| X^{\lambda}(s) - X_{\lambda}(s) \|_H \right ] \leq C_1 \mathrm{e}^{C_2(T-t)} \lambda (1-\lambda) \left ( \|x_1-x_0\|_H^2 + \mathbb{E} \left [ \int_t^T \|a_1(s) - a_0(s) \|_{\Lambda}^2 \mathrm{d}s \right ] \right )
	\end{equation}
    for all $\lambda \in [0,1]$, $x_0,x_1\in H$ and $a_0(\cdot),a_1(\cdot) \in \mathcal{U}_t$.
\end{enumerate}
\end{lemma}

\begin{proof}
We are only going to prove part (ii); the proof of part (i) follows along the same lines.

For $\theta\in [0,1]$, we define
\begin{equation*}
	\begin{cases}
		\bar X_0(\theta) = X^{\lambda}(s) + \theta \lambda (X_0(s) -X_1(s))\\
		\bar X_1(\theta) = X^{\lambda}(s) + \theta (1-\lambda)(X_1(s) -X_0(s))
	\end{cases}
\end{equation*}
and
\begin{equation*}
	\begin{cases}
		\bar a_0(\theta) = a_{\lambda}(s) + \theta \lambda (a_0(s) -a_1(s))\\
		\bar a_1(\theta) = a_{\lambda}(s) + \theta (1-\lambda)(a_1(s) -a_0(s)).
	\end{cases}
\end{equation*}
Note that
\begin{equation}\label{estimateb1}
\begin{split}
	&\lambda b(X_1(s),a_1(s)) +(1-\lambda) b(X_0(s),a_0(s)) - b(X^{\lambda}(s),a_{\lambda}(s))\\
	&= \lambda (1-\lambda) \int_0^1 ( D_x b(\bar X_1(\theta), \bar a_1(\theta)) - D_x b(\bar X_0(\theta), \bar a_0(\theta))) (X_1(s) - X_0(s)) \mathrm{d}\theta\\
	&\quad + \lambda (1-\lambda) \int_0^1 ( D_a b(\bar X_1(\theta), \bar a_1(\theta)) - D_a b(\bar X_0(\theta), \bar a_0(\theta))) (a_1(s) - a_0(s)) \mathrm{d}\theta.
\end{split}
\end{equation}
Therefore, using Assumption \ref{bsigmafrechet}, we obtain
\begin{equation}\label{estimateb}
	\begin{split}
		&\| \lambda b(X_1(s),a_1(s)) +(1-\lambda) b(X_0(s),a_0(s)) - b(X^{\lambda}(s),a_{\lambda}(s)) \|_{H}\\
		&\leq C \lambda (1-\lambda) \left ( \|X_1(s) - X_0(s)\|_H^2 + \|a_1(s) -a_0(s)\|_{\Lambda}^2 \right ).
	\end{split}
\end{equation}
Since $\sigma$ is independent of the control, we obtain similarly
\begin{equation}\label{estimatesigma}
	\| \lambda \sigma(X_1(s)) +(1-\lambda) \sigma(X_0(s)) - \sigma(X^{\lambda}(s)) \|_{L_2(\Xi,H)} \leq C \lambda (1-\lambda) \|X_1(s) - X_0(s)\|_H^2.
\end{equation}
Applying \cite[Proposition 1.166]{fabbri2017}, we obtain
\begin{equation}\label{itosformula4}
	\begin{split}
		&\| X^{\lambda}(s) - X_{\lambda}(s) \|_H^2\\
		&\leq 2 \int_t^s \langle \lambda b(X_1(r),a_1(r)) + (1-\lambda) b(X_0(r),a_0(r)) - b(X^{\lambda}(r),a_{\lambda}(r)) , X^{\lambda}(r) - X_{\lambda}(r) \rangle_H \mathrm{d}r\\
        &\quad + 2 \int_t^s \langle b(X^{\lambda}(r),a_{\lambda}(r)) - b(X_{\lambda}(r),a_{\lambda}(r)),X^{\lambda}(r) - X_{\lambda}(r) \rangle_H \mathrm{d}r\\
		&\quad + 2 \int_t^s \| \lambda \sigma(X_1(r)) + (1-\lambda) \sigma(X_0(r)) - \sigma(X^{\lambda}(r)) \|_{L_2(\Xi,H)}^2 \mathrm{d}r\\
        &\quad + 2 \int_t^s \| \sigma(X^{\lambda}(r)) - \sigma(X_{\lambda}(r)) \|_{L_2(\Xi,H)}^2 \mathrm{d}r\\
		&\quad + 2 \int_t^s \Big \langle X^{\lambda}(r) - X_{\lambda}(r), ( \lambda \sigma(X_1(r)) + (1-\lambda) \sigma(X_0(r)) - \sigma(X^{\lambda}(r)) ) \mathrm{d}W(r)\Big \rangle_H\\
        &\quad +2 \int_t^s \Big \langle X^{\lambda}(r) - X_{\lambda}(r), ( \sigma(X^{\lambda}(r)) - \sigma(X_{\lambda}(r)) ) \mathrm{d}W(r)\Big \rangle_H.
	\end{split}
\end{equation}
Now, first note that for any $t<T^{\prime}\leq T$,
\begin{equation}\label{estimatingb}
\begin{split}
    & \sup_{s\in [t,T^{\prime}]} \left | \int_t^s \langle \lambda b(X_1(r),a_1(r)) + (1-\lambda) b(X_0(r),a_0(r)) - b(X^{\lambda}(r),a_{\lambda}(r)) , X^{\lambda}(r) - X_{\lambda}(r) \rangle_H \mathrm{d}r \right | \\
    &\leq \int_t^{T^{\prime}} \| \lambda b(X_1(s),a_1(s)) + (1-\lambda) b(X_0(s),a_0(s)) - b(X^{\lambda}(s),a_{\lambda}(s)) \|_H \| X^{\lambda}(s) - X_{\lambda}(s) \|_H \mathrm{d}s\\
    &\leq \frac{1}{16} \sup_{s\in [t,T^{\prime}]} \| X^{\lambda}(s) - X_{\lambda}(s) \|^2_H\\
    &\quad + C \left ( \int_t^{T^{\prime}} \| \lambda b(X_1(s),a_1(s)) + (1-\lambda) b(X_0(s),a_0(s)) - b(X^{\lambda}(s),a_{\lambda}(s)) \|_{H} \mathrm{d}s \right )^2.
\end{split}
\end{equation}
Due to the Lipschitz continuity of $b$ and $\sigma$, we have
\begin{equation*}
\begin{split}
    &\sup_{s\in [t,T^{\prime}]} \left | \int_t^s \langle b(X^{\lambda}(r),a_{\lambda}(r)) - b(X_{\lambda}(r),a_{\lambda}(r)),X^{\lambda}(r) - X_{\lambda}(r) \rangle_H \mathrm{d}r \right |\\
    &\leq \frac{1}{16} \sup_{s\in [t,T^{\prime}]} \| X^{\lambda}(s) - X_{\lambda}(s) \|_H^2 + C \left ( \int_t^{T^{\prime}} \| X^{\lambda}(s) - X_{\lambda}(s) \|_H \mathrm{d}s \right )^2
\end{split}
\end{equation*}
and
\begin{equation*}
    \sup_{s\in [t,T^{\prime}]} \int_t^s \| \sigma(X^{\lambda}(r)) - \sigma(X_{\lambda}(r)) \|_{L_2(\Xi,H)}^2 \mathrm{d}r\leq \frac{1}{16} \sup_{s\in [t,T^{\prime}]} \| X^{\lambda}(s) - X_{\lambda}(s) \|_H^2 + C \left ( \int_t^{T^{\prime}} \| X^{\lambda}(s) - X_{\lambda}(s) \|_H \mathrm{d}s \right )^2.
\end{equation*}
By Burkholder--Davis--Gundy inequality,
\begin{equation*}
\begin{split}
    &\mathbb{E} \left [ \sup_{s\in [t,T^{\prime}]} \left | \int_t^s \left \langle X^{\lambda}(r) - X_{\lambda}(r), ( \lambda \sigma(X_1(r)) + (1-\lambda) \sigma(X_0(r)) - \sigma(X^{\lambda}(r))) \mathrm{d}W(r) \right \rangle_H \right |^{\frac12} \right ]\\
    &\leq C \mathbb{E} \left [ \left ( \int_t^{T^{\prime}} \| X^{\lambda}(r) - X_{\lambda}(r) \|_H^2 \| \lambda \sigma(X_1(r)) + (1-\lambda) \sigma(X_0(r)) - \sigma(X^{\lambda}(r)) \|_{L_2(\Xi,H)}^2 \mathrm{d}r \right )^{\frac14} \right ]\\
    &\leq \frac{1}{16} \mathbb{E} \left [ \sup_{s\in [t,T^{\prime}]} \| X^{\lambda}(s) - X_{\lambda}(s) \|_H \right ]\\
    &\quad + C \mathbb{E} \left [ \left ( \int_t^{T^{\prime}} \| \lambda \sigma(X_1(s)) + (1-\lambda) \sigma(X_0(s)) - \sigma(X^{\lambda}(s)) \|_{L_2(\Xi,H)}^2 \mathrm{d}s \right )^{\frac12} \right ],
\end{split}
\end{equation*}
as well as
\begin{equation*}
\begin{split}
    &\mathbb{E} \left [ \sup_{s\in [t,T^{\prime}]} \left | \int_t^s \Big \langle X^{\lambda}(r) - X_{\lambda}(r), ( \sigma(X^{\lambda}(r)) - \sigma(X_{\lambda}(r)) ) \mathrm{d}W(r)\Big \rangle_H \right |^{\frac12} \right ]\\
    &\leq C \mathbb{E} \left [ \left ( \int_t^{T^{\prime}} \| X^{\lambda}(s) - X_{\lambda}(s) \|_H^2 \| \sigma(X^{\lambda}(s)) - \sigma(X_{\lambda}(s)) \|_{L_2(\Xi,H)}^2 \mathrm{d}s \right )^{\frac14} \right ]\\
    &\leq \frac{1}{16} \mathbb{E} \left [ \sup_{s\in [t,T^{\prime}]} \| X^{\lambda}(s) - X_{\lambda}(s) \|_H \right ] + C \mathbb{E} \left [ \int_t^{T^{\prime}} \| X^{\lambda}(s) - X_{\lambda}(s) \|_{H} \mathrm{d}s \right ].
\end{split}
\end{equation*}
Thus, taking the absolute value and square root of both sides of inequality \eqref{itosformula4}, taking the supremum over $[t,T^{\prime}]$, taking the expectation and using the previous five inequalities as well as \eqref{estimateb} and \eqref{estimatesigma}, we obtain
\begin{equation*}
\begin{split}
	\mathbb{E} \left [ \sup_{s\in [t,T^{\prime}]} \left \|X^{\lambda}(s) - X_{\lambda}(s)\right \|_H \right ] &\leq C \mathbb{E} \left [ \int_t^{T^{\prime}} \sup_{r\in [t,s]} \| X^{\lambda}(r) - X_{\lambda}(r) \|_H \mathrm{d}s \right ]\\
    &\quad + C \lambda (1-\lambda) \mathbb{E} \left [ (T^{\prime}-t)\sup_{s\in [t,T^{\prime}]} \|X_1(s) - X_0(s)\|_H^2 + \int_t^{T^{\prime}} \|a_1(s) -a_0(s)\|_{\Lambda}^2 \mathrm{d}s  \right ].
\end{split}
\end{equation*}
Applying Lemma \ref{estimatex1x0one}(ii) and Gr\"onwall's inequality concludes the proof.
\end{proof}

\begin{assumption}\label{lgsemiconcave}
    \begin{enumerate}[label=(\roman*)]
        \item Let $l(\cdot,a)$ be semiconcave uniformly in $a\in \Lambda_0$.
        \item Let $g$ be semiconcave.
    \end{enumerate}
\end{assumption}

\begin{theorem}\label{th:semiconc}
    Let Assumptions \ref{bsigmalipschitzfirstvariable}, \ref{lglipschitzfirstvariable}, \ref{bsigmafrechetfirstvariable} and \ref{lgsemiconcave} be satisfied. Then, for every $t\in[0,T]$, the function $V(t,\cdot)$ is semiconcave with semiconcavity constant $C_1 \mathrm{e}^{C_2(T-t)}$ for some $C_1,C_2\geq 0$ independent of $t,T$.
\end{theorem}

\begin{proof}
    The proof is analogous to the finite dimensional case, which can be found, e.g., in \cite[Chapter 4, Proposition 4.5]{yong1999}.
    
	It is sufficient to show that for every $a(\cdot)\in \mathcal{U}_t$, the function $J(t,\cdot;a(\cdot))$ is semiconcave with the semiconcavity constant independent of $a(\cdot)$. Let $X_0$ and $X_1$ be given by \eqref{x1x0differentcontrols} with $a_0(\cdot) = a_1(\cdot) = a(\cdot)$, and let $x_{\lambda}$, $X_{\lambda}$ and $X^{\lambda}$ be given by \eqref{lambdadefinition}. Note that $a_{\lambda}(\cdot) = a(\cdot)$ in this case. Then
    \begin{equation*}
    \begin{split}
        &\lambda J(t,x_1;a(\cdot)) + (1-\lambda) J(t,x_0;a(\cdot)) - J(t,x_{\lambda};a(\cdot))\\
        &= \mathbb{E} \left [ \int_t^T \lambda l(X_1(s),a(s)) + (1-\lambda) l(X_0(s),a(s)) - l(X_{\lambda}(s),a(s)) \mathrm{d}s \right ]\\
        &\quad + \mathbb{E} \left [ \lambda g(X_1(T)) + (1-\lambda) g(X_0(T)) - g(X_{\lambda}(T)) \right ].
    \end{split}
    \end{equation*}
    Due to the semiconcavity of $l(\cdot,a)$, we have
    \begin{equation*}
        \lambda l(X_1(s),a(s)) + (1-\lambda) l(X_0(s),a(s)) - l(X^{\lambda}(s),a(s)) \leq C \lambda(1-\lambda) \|X_1(s) - X_0(s) \|_H^2
    \end{equation*}
    and the corresponding inequality holds for $g$. Thus,
    \begin{equation*}
    \begin{split}
        &\lambda J(t,x_1;a(\cdot)) + (1-\lambda)J(t,x_0;a(\cdot)) - J(t,x_{\lambda};a(\cdot))\\
        &\leq C\lambda(1-\lambda) \mathbb{E} \left [ \int_t^T \| X_1(s) - X_0(s) \|_H^2 \mathrm{d}s \right ] + \mathbb{E} \left [ \int_t^T |l(X^{\lambda}(s),a(s)) - l(X_{\lambda}(s),a(s)) | \mathrm{d}s \right ]\\
        &\quad + C\lambda (1-\lambda) \mathbb{E} \left [ \| X_1(T) - X_0(T) \|_H^2 \right ] + \mathbb{E} \left [ | g(X^{\lambda}(T)) - g(X_{\lambda}(T)) | \right ].
    \end{split}
    \end{equation*}
    Using the Lipschitz continuity of $l$ and $g$ as well as Lemma \ref{estimatexlambdasamecontrol}(i), we obtain
    \begin{equation*}
        \lambda J(t,x_1;a(\cdot)) + (1-\lambda) J(t,x_0;a(\cdot)) - J(t,x_{\lambda};a(\cdot)) \leq C_1 \mathrm{e}^{C_2(T-t)} \lambda (1-\lambda) \|x_1 - x_0\|^2_H.
    \end{equation*}
   This proves the semiconcavity of $J(t,\cdot;a(\cdot))$ with the required constant.
\end{proof}

\subsection{Semiconvexity}\label{semiconvexity}

In this subsection, we derive semiconvexity of the value function under three different sets of assumptions.

\subsubsection{Case 1: Uniform Convexity of Running Cost in The Control Variable}

\begin{assumption}\label{lgsemiconvex}
	\begin{enumerate}[label=(\roman*)]
		\item There exist constants $C,\nu \geq 0$ such that the map
		\begin{equation*}
			H\times\Lambda_0 \ni (x,a)\mapsto l(x,a) + C\|x\|_H^2 - \nu \|a\|^2_{\Lambda}
		\end{equation*}
		is convex.
    \item Let $g:H\to\mathbb{R}$ be semiconvex.
	\end{enumerate}
\end{assumption}

\begin{theorem}\label{convexity1}
	Let Assumptions \ref{bsigmalipschitz}, \ref{lglipschitzfirstvariable}, \ref{bsigmafrechet} and \ref{lgsemiconvex} be satisfied. Then there exists a constant $\nu_0$ depending only on the data of the problem such that if $\nu\geq \nu_0$, then $V(t,\cdot)$ is semiconvex with constant $C_1\mathrm{e}^{C_2T}$.
\end{theorem}

\begin{proof}
    Let $\varepsilon>0$, $x_0,x_1\in H$. Let $X_0(s),X_1(s),x_{\lambda},X^{\lambda}(s),X_{\lambda}(s)$ be as in \eqref{x1x0differentcontrols} and \eqref{lambdadefinition}, respectively, where $a_0(\cdot)=a_0^{\varepsilon}(\cdot), a_1(\cdot)=a_1^{\varepsilon}(\cdot)$, are chosen such that
    \begin{equation*}
        V(t,x_0)\geq J(t,x; a_0^{\varepsilon}(\cdot))-\varepsilon,\quad V(t,x_1)\geq J(t,x; a_1^{\varepsilon}(\cdot))-\varepsilon
    \end{equation*}
    and for $0\leq\lambda\leq 1$ set $a_\lambda^{\varepsilon}(\cdot)= \lambda a_1^{\varepsilon}(\cdot)+(1-\lambda)a_0^{\varepsilon}(\cdot)$. Then
    \begin{equation}\label{semiconvexityinequality}
    \begin{split}
        &\varepsilon+\lambda V(t,x_1)+(1-\lambda )V(t,x_0)-V(t,x_\lambda)\\
        &\geq \lambda J(t,x_1;a_1^{\varepsilon}(\cdot)) + (1-\lambda) J(t,x_0; a_0^{\varepsilon}(\cdot)) - J(t,x_{\lambda};a_{\lambda}^{\varepsilon}(\cdot))
        \\
        &= \mathbb{E} \left [ \int_t^T \lambda l(X_1(s),a_1^{\varepsilon}(s)) + (1-\lambda)l(X_0(s),a_0^{\varepsilon}(s)) - l(X_{\lambda}(s),a_{\lambda}^{\varepsilon}(s)) \mathrm{d}s \right ]\\
        &\quad + \mathbb{E} \left [ \lambda g(X_1(T)) + (1-\lambda)g(X_0(T)) - g(X_{\lambda}(T)) \right ]\\
        &= \mathbb{E} \left [ \int_t^T \lambda l(X_1(s),a_1^{\varepsilon}(s)) + (1-\lambda)l(X_0(s),a_0^{\varepsilon}(s)) - l(X^{\lambda}(s),a_{\lambda}^{\varepsilon}(s)) \mathrm{d}s \right ]\\
        &\quad + \mathbb{E} \left [ \int_t^T l(X^{\lambda}(s),a_{\lambda}^{\varepsilon}(s)) - l(X_{\lambda}(s),a_{\lambda}^{\varepsilon}(s)) \mathrm{d}s \right ]\\
        &\quad + \mathbb{E} \left [ \lambda g(X_1(T)) + (1-\lambda)g(X_0(T)) - g(X^{\lambda}(T)) + g(X^{\lambda}(T)) - g(X_{\lambda}(T)) \right ].
    \end{split}
    \end{equation}
    Due to Assumption \ref{lgsemiconvex}, we have
    \begin{equation*}
    \begin{split}
        &\lambda l(X_1(s),a_1^{\varepsilon}(s)) + (1-\lambda)l(X_0(s),a_0^{\varepsilon}(s)) - l(X^{\lambda}(s),a_{\lambda}^{\varepsilon}(s))\\
        &\geq - C\lambda(1-\lambda)\|X_1(s)-X_0(s)\|_H^2 + \nu \lambda (1-\lambda) \|a_1^{\varepsilon}(s) - a_0^{\varepsilon}(s) \|_{\Lambda}^2
    \end{split}
    \end{equation*}
    and
    \begin{equation*}
        \lambda g(X_1(T)) + (1-\lambda)g(X_0(T)) - g(X^{\lambda}(T)) \geq -C\lambda(1-\lambda)\|X_1(T)-X_0(T)\|_H^2.
    \end{equation*}
    Therefore, we derive from \eqref{semiconvexityinequality} using Assumption \ref{lglipschitzfirstvariable}, \eqref{eq:strongest} and \eqref{eq:xlambda}
    \begin{equation*}
    \begin{split}
        &\varepsilon+\lambda V(t,x_1)+(1-\lambda )V(t,x_0)-V(t,x_\lambda)\\
        &\geq - C_1\mathrm{e}^{C_2(T-t)} \lambda(1-\lambda) \left ( \|x_1-x_0\|_H^2 + \mathbb{E} \left [ \int_t^T \|a_1^{\varepsilon}(s) - a_0^{\varepsilon}(s) \|_{\Lambda}^2 \mathrm{d}s \right ] \right )\\
        &\quad + \nu \lambda (1-\lambda) \mathbb{E} \left [ \int_t^T \|a_1^{\varepsilon}(s) - a_0^{\varepsilon}(s) \|_{\Lambda}^2 \mathrm{d}s \right ] \geq - C_1\mathrm{e}^{C_2(T-t)} \lambda(1-\lambda) \|x_1-x_0\|_H^2
    \end{split}
    \end{equation*}
   if $\nu\geq \nu_0=C_1\mathrm{e}^{C_2T}$. This implies that the function $V(t,\cdot)$ is semiconvex.
\end{proof}

\subsubsection{Case 2: Linear State Equation and Convex Costs}

\begin{assumption}\label{blinearlgconvex}
	\begin{enumerate}[label=(\roman*)]
		\item Let $b:H\times\Lambda\to H$ and $\sigma: H \times \Lambda \to L_2(\Xi,H)$ be affine and continuous.
		\item Let $l:H\times\Lambda_0\to\mathbb{R}$ and $g:H\to\mathbb{R}$ be convex.
	\end{enumerate}
\end{assumption}

\begin{theorem}\label{convexity2}
	Let Assumptions \ref{lglipschitzfirstvariable} and  \ref{blinearlgconvex} be satisfied. Then, for every $t\in [0,T]$, the function $V(t,\cdot)$ is convex.
\end{theorem}
Note that in contrast to Theorem \ref{convexity1}, in Theorem \ref{convexity2} we do not require strong convexity in the control of the cost function $l$ of the type used in Assumption \ref{lgsemiconvex}(i). Also the full Assumption \ref{lglipschitzfirstvariable} is not necessary but we use it here since it guarantees that the value function is well defined.

\begin{proof}
    We use the same notation as in the proof of Theorem \ref{convexity1}. Note that due to the linearity of the state equation, we have $X^{\lambda}(s) = X_{\lambda}(s)$. Due to Assumption \ref{blinearlgconvex}, we have
    \begin{equation*}
    \begin{split}
        &\lambda J(t,x_1;a_1^{\varepsilon}(\cdot)) + (1-\lambda) J(t,x_0; a_0^{\varepsilon}(\cdot)) - J(t,x_{\lambda};a_{\lambda}^{\varepsilon}(\cdot))\\
        &= \mathbb{E} \left [ \int_t^T \lambda l(X_1(s),a_1^{\varepsilon}(s)) + (1-\lambda)l(X_0(s),a_0^{\varepsilon}(s)) - l(X_{\lambda}(s),a_{\lambda}^{\varepsilon}(s)) \mathrm{d}s \right ]\\
        &\quad + \mathbb{E} \left [ \lambda g(X_1(T)) + (1-\lambda)g(X_0(T)) - g(X_{\lambda}(T)) \right ]
        \geq 0,
    \end{split}
    \end{equation*}
    from which the convexity of $V(t,\cdot)$ follows.
\end{proof}

\subsubsection{Case 3: Comparison for the State Equation}\label{section:case3}

In this subsection we show the convexity of the value function by applying comparison principles for SPDEs. Throughout this section we set the Hilbert space $H=L^2(\mathcal{O})$, for some domain $\mathcal{O}\subset \mathbb{R}^d$, $d\in\mathbb{N}$. Moreover, for $x_1,x_2\in L^2(\mathcal{O})$, $x_1\leq x_2$ is understood in the almost everywhere sense. Finally, for $x \in L^2(\mathcal{O})$ we denote by $x_+$ its positive part, that is $x_+(\xi)=x(\xi)_+$, $\xi\in \mathcal{O}$, where for $a\in\mathbb{R}$, $a_+ := \max\{a,0\}$.

\paragraph{Stochastic Partial Differential Equations (Monotonicity)}

To state a comparison result we impose the following assumptions.
\begin{assumption}\label{comparisonprinciple}
    Let $\mathrm{e}^{sA}$ be positivity preserving for every $s\ge 0$, i.e., for $x\in H$ with $x\geq 0$ it holds $\mathrm{e}^{sA}x \geq 0$.
\end{assumption}

\begin{assumption}\label{comp:monotonicity}
\begin{enumerate}[label=(\roman*)]
\item Let $\bar b:\Omega\times[t,T]\times H\to H$ be $\textit{Prog}_{[t,T]}\otimes {\mathscr B}(H)/ {\mathscr B}(H)$ measurable (where $\textit{Prog}_{[t,T]} $is the $\sigma$-field of progressively measurable sets, see e.g., \cite[page 44]{RY} or \cite[bottom of page 23]{fabbri2017}) and ${\mathscr B}(H)$ is the Borel $\sigma$-field in $H$) 
and assume moreover that (suppressing as usual the first variable)
\begin{equation}\label{assumptionstar}
    \| [ \bar b(s,x_1) - \bar b(s,x_2) ]_+ \|_H \leq C \| [x_1-x_2]_+ \|_H
\end{equation}
for all $s\in [t,T]$ and $x_1,x_2\in H$.
\item Let $\bar{\sigma}:\Omega \times [t,T] \to L_2(\Xi,H)$ be $\textit{Prog}_{[t,T]}/ {\mathscr B}(L_2(\Xi,H))$ measurable.
\end{enumerate}
\end{assumption}
Under these assumptions, we have the following comparison result stated below in Theorem \ref{comparisonSPDEs1}. We emphasize that in Theorem \ref{comparisonSPDEs1} we do not assume that mild solutions $X_i$, $i=1,2$, of equation \eqref{eq:mildxi} are unique. They only have to satisfy the conditions imposed in the definition of mild solutions, see \cite[Definition 1.119]{fabbri2017}.
\begin{theorem}\label{comparisonSPDEs1}
    Let Assumptions \ref{comparisonprinciple} and \ref{comp:monotonicity} be satisfied and let $0\leq t<T$. Let $X_i$, $i=1,2$, be a mild solution of
    \begin{equation}\label{eq:mildxi}
    \begin{cases}
        \mathrm{d}X(s) = [AX(s) + \bar b(s,X(s))+f_i(s) ] \mathrm{d}s + \bar{\sigma}(s) \mathrm{d}W(s),\quad s\in [t,T]\\
        X(t)=x_i\in L^2(\mathcal{O}),
    \end{cases}
    \end{equation}
    where $f_i:\Omega\times [t,T] \to H$ are $\textit{Prog}_{[t,T]}/\mathscr{B}(H)$ measurable and $f_i \in L^1(t,T;H)$ $\mathbb{P}$--almost surely, $i=1,2$. Furthermore, let
    \begin{equation}\label{assumptioncomparison}
    \begin{cases}
        x_1 \geq x_2,\\
        f_1(s) \geq f_2(s).
    \end{cases}
    \end{equation}
    Then $X_1(s) \geq X_2(s)$ for all $s\in [t,T]$.
\end{theorem}
\begin{proof}
Since $A$ generates a positivity preserving semigroup and by assumption \eqref{assumptioncomparison}, we have
\begin{equation*}
    X_2(s) - X_1(s)\leq \int_t^s \mathrm{e}^{(s-r)A} [ \bar b(r,X_2(r)) - \bar b(r,X_1(r)) ] \mathrm{d}r \leq \int_t^s \mathrm{e}^{(s-r)A} [ \bar b(r,X_2(r)) - \bar b(r,X_1(r)) ]_+ \mathrm{d}r
\end{equation*}
and since the right hand side is a non-negative function,
\begin{equation*}
    [ X_2(s) - X_1(s) ]_+ \leq \int_t^s \mathrm{e}^{(s-r)A} [ \bar b(r,X_2(r)) - \bar b(r,X_1(r)) ]_+ \mathrm{d}r.
\end{equation*}
Thus, by \eqref{assumptionstar}, 
\begin{equation*}
\begin{split}
    \left \| [ X_2(s) - X_1(s) ]_+ \right \|_H &\leq \int_t^s \left \| \mathrm{e}^{(s-r)A} [ \bar b(r,X_2(r)) - \bar b(r,X_1(r)) ]_+ \right \|_H \mathrm{d}r\leq C \int_t^s \left \| [ X_2(r) - X_1(r) ]_+ \right \|_H \mathrm{d}r.
\end{split}
\end{equation*}
Gr\"onwall's inequality now yields $X_2(s) \leq X_1(s)$.
\end{proof}

\begin{assumption}\label{comparisonprinciple2}
    \begin{enumerate}[label=(\roman*)]
        \item Let $l:H \times \Lambda_0\to \mathbb{R}$ be convex.
        \item Let $g:H \to \mathbb{R}$ be convex.
        \item Let $l(\cdot,a)$ and $g$ satisfy
        \begin{equation}\label{eq:lgdecreasing}
        \begin{cases}
            l(x_1,a) \geq l(x_2,a),\\
            g(x_1) \geq g(x_2),
        \end{cases}
        \end{equation}
        for all $x_1,x_2\in H$ such that $x_1\leq x_2$ and $a\in \Lambda_0$.
    \end{enumerate}
\end{assumption}

Using Theorem \ref{comparisonSPDEs1} we now obtain the following result.
\begin{theorem}\label{convexity3}
Let Assumptions \ref{bsigmalipschitzfirstvariable}, \ref{lglipschitzfirstvariable}, \ref{comparisonprinciple} and \ref{comparisonprinciple2} be satisfied. Assume in addition that $b$ satisfies
\begin{equation}\label{bconvexity}
    \lambda b(x_1,a_1)+(1-\lambda)b(x_0,a_0) \leq b(\lambda x_1+(1-\lambda)x_0,\lambda a_1+(1-\lambda)a_0)
\end{equation}
for all $x_0,x_1\in H, a_0,a_1\in \Lambda_0, \lambda\in[0,1]$, and that $\sigma$ is constant. Then $V(t,\cdot)$ is convex for $0\leq t\leq T$.
\end{theorem}

\begin{proof}
    We use the same notation as in the proof of Theorem \ref{convexity1}. First note that 
    \begin{equation*}
        \mathrm{d}X_{\lambda}(s) = [ AX_{\lambda}(s) + b(X_{\lambda}(s),a^{\varepsilon}_{\lambda}(s)) ] \mathrm{d}s + \sigma \mathrm{d}W(s)
    \end{equation*}
    and
    \begin{equation*}
    \begin{split}
        &\mathrm{d}X^{\lambda}(s)\\
        & = [ A X^{\lambda}(s) + \lambda b(X_1(s),a^{\varepsilon}_1(s)) + (1-\lambda) b(X_0(s),a^{\varepsilon}_0(s)) ]
        \mathrm{d}s + \sigma \mathrm{d}W(s)\\
       &= [ A X^{\lambda}(s) + b(X^{\lambda}(s),a^{\varepsilon}_{\lambda}(s))+\lambda b(X_1(s),a^{\varepsilon}_1(s)) + (1-\lambda) b(X_0(s),a^{\varepsilon}_0(s))-b(X^{\lambda}(s),a^{\varepsilon}_{\lambda}(s)) ]
        \mathrm{d}s + \sigma \mathrm{d}W(s).
    \end{split}
    \end{equation*}
    Due to \eqref{bconvexity}, we have
    \begin{equation*}
    \lambda b(X_1(s),a^{\varepsilon}_1(s)) + (1-\lambda) b(X_0(s),a^{\varepsilon}_0(s)) -b(X^{\lambda}(s),a^{\varepsilon}_{\lambda}(s)) \leq 0.
    \end{equation*}
    Thus, setting
    \begin{equation*}
   f_1(s)=0, \quad f_2(s) = \lambda b(X_1(s),a^{\varepsilon}_1(s)) + (1-\lambda) b(X_0(s),a^{\varepsilon}_0(s)) - b(X^{\lambda}(s),a^{\varepsilon}_{\lambda}(s)),
    \end{equation*}
  we have $f_1(s)\geq f_2(s)$. Therefore, we can apply Theorem \ref{comparisonSPDEs1} to obtain
\begin{equation}\label{comparisonresult}
      X_{\lambda}(s) \geq X^{\lambda}(s).
    \end{equation}    
    Now, since $l$ and $g$ are decreasing in the $x$-variable and convex, we have
    \begin{equation*}
    \begin{split}
        V(t,x_{\lambda})
        &\leq J(t,x_{\lambda};a^{\varepsilon}_{\lambda}(\cdot))
        =\mathbb{E} \left [ \int_t^T l(X_{\lambda}(s),a^{\varepsilon}_{\lambda}(s))\mathrm{d}s + g(X_{\lambda}(T)) \right ]\\
        &\leq \mathbb{E} \left [ \int_t^T l(X^{\lambda}(s),a^{\varepsilon}_{\lambda}(s))\mathrm{d}s + g(X^{\lambda}(T)) \right ]\\
        &\leq \mathbb{E} \left [ \int_t^T (\lambda l(X_1(s),a^{\varepsilon}_1(s)) + (1-\lambda) l(X_0(s),a^{\varepsilon}_0(s))) \mathrm{d}s + \lambda g(X_1(T)) + (1-\lambda) g(X_0(T)) \right ]\\
        &= \lambda J(t,x_1;a^{\varepsilon}_1(\cdot)) + (1-\lambda) J(t,x_0;a^{\varepsilon}_0(\cdot))\\
        &\leq \lambda V(t,x_1) + (1-\lambda) V(t,x_0) + \varepsilon,
    \end{split}
    \end{equation*}
    which concludes the proof.
\end{proof}

\begin{remark}
    Note that the conclusion of Theorem \ref{convexity3} still holds if the inequality in \eqref{bconvexity} is reversed (that is $b$ is ``convex'')  and \eqref{eq:lgdecreasing} is satisfied whenever $x_1\geq x_2$ (that is $l$ and $g$ are ``non-decreasing'').
\end{remark}

\paragraph{Stochastic Partial Differential Equations (Nemytskii Type)}

In this part, we apply Theorem \ref{comparisonSPDEs1} to derive a comparison result for SPDEs with Nemytskii type coefficients. Similar comparison results were proven for example in \cite{kotelenez1992,manthey1999,milian2002}. However, none of these results applies directly to our situation.

\begin{theorem}\label{comparisonSPDEs}
Let Assumption \ref{comparisonprinciple} hold. Let $\mathfrak{b}:\Omega\times[t,T]\times \mathbb{R} \to \mathbb{R}$ be $\textit{Prog}_{[t,T]}\otimes {\mathscr B}(\mathbb{R})/ {\mathscr B}(\mathbb{R})$ measurable (where ${\mathscr B}(\mathbb{R})$ is the Borel $\sigma$-field in $\mathbb{R}$) and let there be a constant $C_L$ such that (suppressing as usual the first variable)
    \begin{equation*}
        |\mathfrak{b}(s,\mathrm{x}_1) - \mathfrak{b}(s,\mathrm{x}_2) | \leq C_L |\mathrm{x}_1-\mathrm{x}_2|.
    \end{equation*}
    for all $s\in [t,T]$ and $\mathrm{x},\mathrm{x}_1,\mathrm{x}_2\in \mathbb{R}$. Furthermore, assume that $\mathfrak{b}(\cdot,0)\in L^1(t,T)$ $\mathbb{P}$--almost surely. Let $b$ denote the Nemytskii operator associated with $\mathfrak{b}$, that is $b(s,x)(\xi):=\mathfrak{b}(s,x(\xi))$ for $s\in[t,T],x\in H$ and $\xi\in\mathcal{O}$. Furthermore, let $f_1,f_2: \Omega\times [t,T] \to H$ be $\textit{Prog}_{[t,T]}/ {\mathscr B}(H)$ measurable with $f_1,f_2\in L^1(t,T;H)$ $\mathbb{P}$--almost surely and $\sigma:\Omega \times [t,T] \to L_2(\Xi,H)$ be $\textit{Prog}_{[t,T]}/ {\mathscr B}(L_2(\Xi,H))$ measurable with $\sigma \in L^2(t,T;L_2(\Xi,H))$ $\mathbb{P}$--almost surely. Let $X_i$ be a mild solution of
    \begin{equation*}
    \begin{cases}
        \mathrm{d}X(s) = [AX(s) + b(s,X(s)) +f_i(s) ] \mathrm{d}s + \sigma(s) \mathrm{d}W(s),\quad s\in [t,T]\\
        X(t)=x_i\in L^2(\mathcal{O}),
    \end{cases}
    \end{equation*}
    $i=1,2$. Furthermore, let
    \begin{equation*}
    \begin{cases}
        x_1\geq x_2,\\
        f_1(s) \geq f_2(s).
    \end{cases}
    \end{equation*}
    Then $X_1(s) \geq X_2(s)$ for all $s\in [t,T]$.
\end{theorem}

\begin{proof}[Proof of Theorem \ref{comparisonSPDEs}]
    Note that for $C\in\mathbb{R}$, the process $Y_i(s) := \mathrm{e}^{C(s-t)} X_i(s)$, $s\in [t,T]$, is a mild solution of
    \begin{equation*}
        \mathrm{d}Y(s) = [ AY(s) + CY(s) + \mathrm{e}^{C(s-t)} b(s,\mathrm{e}^{-C(s-t)}Y(s)) + \mathrm{e}^{C(s-t)} f_i(s) ] \mathrm{d}s + \mathrm{e}^{C(s-t)} \sigma(s) \mathrm{d}W(s).
    \end{equation*}
 We set $C=C_L$, where $C_L$ is the Lipschitz constant of $\mathfrak{b}$. Then
    \begin{equation*}
        \bar b(s,x) := C_L x + \mathrm{e}^{C_Ls} b(s,\mathrm{e}^{-C_Ls} x)\quad \text{and}\quad \bar \sigma(s) := \mathrm{e}^{C(s-t)} \sigma(s)
    \end{equation*}
    satisfies Assumption \ref{comp:monotonicity}. Indeed, note that due to the Lipschitz continuity of $\mathfrak{b}$, we have
    \begin{equation*}
    \begin{split}
        C_L\mathrm{x}_1 + \mathrm{e}^{C_Ls} \mathfrak{b}(s,\mathrm{e}^{-C_Ls}\mathrm{x}_1) - C_L\mathrm{x}_2 - \mathrm{e}^{C_Ls} \mathfrak{b}(s,\mathrm{e}^{-C_Ls}\mathrm{x}_2) \leq \begin{cases}
            2C_L(\mathrm{x}_1-\mathrm{x}_2),& \text{if}\; \mathrm{x}_1\geq \mathrm{x}_2\\
            0,& \text{if}\; \mathrm{x}_1\leq \mathrm{x}_2,
        \end{cases}
    \end{split}
    \end{equation*}
    and therefore,
    \begin{equation*}
        \left [ C_L\mathrm{x}_1 + \mathrm{e}^{C_Ls} \mathfrak{b}(s,\mathrm{e}^{-C_Ls}\mathrm{x}_1) - C_L\mathrm{x}_2 - \mathrm{e}^{C_Ls} \mathfrak{b}(s,\mathrm{e}^{-C_Ls}\mathrm{x}_2) \right ]_+ \leq 2C_L [\mathrm{x}_1 - \mathrm{x}_2 ]_+
    \end{equation*}
    from which \eqref{assumptionstar} follows.
    Thus, we can apply Theorem \ref{comparisonSPDEs1}, which concludes the proof.
\end{proof}

\begin{assumption}\label{comparisonprinciple3}
    The function $\mathfrak{b}: \mathbb{R} \times\Lambda_0\to \mathbb{R}$ is continuous and there exists a constant $C>0$ such that
    \begin{equation*}
    \begin{cases}
        |\mathfrak{b}(\mathrm{x}_1,a) - \mathfrak{b}(\mathrm{x}_2,a) | \leq C |\mathrm{x}_1 - \mathrm{x}_2|\\
    	| \mathfrak{b}(\mathrm{x},a) | \leq C(1 + |\mathrm{x}| + \| a \|_{\Lambda} )
    \end{cases}
    \end{equation*}
    for all $\mathrm{x},\mathrm{x}_1,\mathrm{x}_2\in \mathbb{R}$, $a\in \Lambda_0$, and
    \begin{equation*}
        \lambda \mathfrak{b}(\mathrm{x}_1,a_1) +(1-\lambda) \mathfrak{b}(\mathrm{x}_0,a_0) \leq \mathfrak{b}(\lambda \mathrm{x}_1 +(1-\lambda)\mathrm{x}_0,\lambda a_1+(1-\lambda)a_0),
    \end{equation*}
    for all $\lambda \in [0,1]$, $\mathrm{x}_0,\mathrm{x}_1\in \mathbb{R}$ and $a_0,a_1\in \Lambda_0$.
\end{assumption}

We consider the control problem \eqref{costfunctional}--\eqref{state}, where $b:H\times \Lambda_0 \to H$ is given by the Nemytskii operator associated with $\mathfrak{b}$, i.e.,
\begin{equation*}
    b(x,a)(\xi) := \mathfrak{b}(x(\xi),a)
\end{equation*}
for $x\in H$, $a\in \Lambda_0$ and $\xi\in\mathcal{O}$.

\begin{theorem}\label{th:convcomparison1}
Let Assumptions \ref{lglipschitzfirstvariable}, \ref{comparisonprinciple}, \ref{comparisonprinciple2} and \ref{comparisonprinciple3} be satisfied, and let $\sigma$ be constant. Then $V(t,\cdot)$ is convex.
\end{theorem}

The proof follows along the same lines as the proof of Theorem \ref{convexity3}.

\begin{remark}\label{remark_spde_nemytskii}
\begin{itemize}
    \item[(i)] The method used in Section \ref{semiconcavity} to prove semiconcavity of the value function does not directly apply to control problems governed by SPDEs of Nemystkii type. The reason is that $C^{1,1}$ regularity of the function $\mathfrak{b}$ does not translate to $C^{1,1}$ regularity of the associated Nemytskii operator. Nevertheless, the method can be adapted and semiconcavity can still be proven under reasonable assumptions, see Section \ref{example_reaction_diffusion}.
    \item[(ii)] In this whole subsection, the restriction to additive noise is crucial only in the proofs of the comparison theorems. If the comparison property \eqref{comparisonresult}, together with the assumptions imposed on $l$ and $g$, holds, the same arguments would give convexity of the value function for problems with multiplicative noise.
\end{itemize}    
\end{remark}

\subsection{$C^{1,1}$ Regularity of the Value Function}\label{sec:C11reg}

It is well known that if a uniformly continuous function $v:H\to\mathbb{R}$ is semiconvex and semiconcave, then $v\in C^{1,1}(H)$. The proof of this result in a Hilbert space can be found for instance in \cite{lasry1986}. Thus in each of the cases of Subsection \ref{semiconvexity}, in combination with the semiconcavity of the value function proved in Subsection \ref{semiconcavity}, we obtain that $V(t,\cdot)$ is $C^{1,1}$ on $H$ for all $t\in [0,T]$ (for SPDEs of Nemystkii type, see Remark \ref{remark_spde_nemytskii}(i) as well as Section \ref{example_reaction_diffusion}). We point out that this $C^{1,1}$ regularity result for the value function is applicable also to deterministic problems, i.e., when $\sigma=0.$

\section{Optimal Synthesis}\label{sec:optsynt}
In this section, we show how to use viscosity solution methods and the fact that the value function is $C^{1,1}$  in the space variable to construct optimal feedback controls. Since we only have one derivative of the value function, in order to perform optimal synthesis, we need to assume that the noise coefficient $\sigma$ is independent of the control variable.
\begin{assumption}\label{hp:sigma_independent_control}
Let $\sigma(x,a) = \sigma(x)$ be independent of the control and let there be a constant $C>0$ such that
\begin{equation*}
    \| \sigma(x) - \sigma(x^{\prime}) \|_{L_2(\Xi,H)} \leq C \|x-x^{\prime}\|_{-1}
\end{equation*}
for all $x,x^{\prime}\in H$.
\end{assumption}
In this case, the HJB equation is given by
\begin{equation}\label{HJBsemilinear}
\begin{cases}
	v_t + \langle Ax, Dv \rangle_H + \frac 1 2 \text{Tr} [ \sigma(x) \sigma^{\ast}(x) D^2 v ] + \mathcal{H}(x,Dv) =0\\
	v(T,\cdot) = g,
\end{cases}
\end{equation}
where the Hamiltonian $\mathcal{H}:H\times H \to \mathbb{R}$ is given by
\begin{equation*}
\mathcal{H}(x,p)=\inf_{a\in \Lambda_0}
\mathcal{F}(x,p,a),\quad \mathcal{F}(x,p,a) := \langle p,b(x,a) \rangle_H + l(x,a)
\end{equation*}
for all $x,p \in H$. For $m>0$, we define
\begin{equation*}
    \mathcal{H}_m(x,p):=\inf_{a\in \Lambda_0, \|a\|_\Lambda \leq m} \mathcal{F}(x,p,a)
\end{equation*}
for all $x,p \in H$. Since the coefficients of our control problem may not be bounded in $a \in \Lambda_0$ we assume the following.
\begin{assumption}\label{coercivity}
For every $R>0$ there is $m_R>0$ such that
\begin{equation*}
\mathcal{H}(x,p)=\mathcal{H}_{m_R}(x,p)
\end{equation*}
for all $x,p\in H,\|p\|_H\leq R$.
\end{assumption}
\begin{remark}\label{rem:H_attains_inf_bounded_set_H}
  This assumption is satisfied for instance if $b(x,a)=b_1(x)+b_2(x,a)$, where  $b_1,b_2$ satisfy Assumption \ref{bsigmalipschitzfirstvariable}(i) and
\begin{equation*}
\|b_2(x,a)\|_H\leq C(1+\|a\|_{\Lambda})
\end{equation*}
for all $x\in H, a\in\Lambda_0$, and
\begin{equation}\label{eq:superlinear_growth_l}
\lim_{\|a\|_{\Lambda} \to+ \infty}\frac{l(x,a)}{\|a\|_{\Lambda}}=+\infty
\end{equation}
uniformly in $x \in H$.  
\end{remark}

Let us recall the definition of a viscosity solution of \eqref{HJBsemilinear}, see\cite[Definition 3.35]{fabbri2017}. 
\begin{definition}\label{def:test_functions}
A function $\psi:(0,T)\times H\to \mathbb{R}$ is a test function if $\psi=\varphi+h$, where:
\begin{itemize}
\item[(i)] 
$\varphi \in C^{1,2}((0,T)\times H), \varphi$ is $B$--lower semicontinuous and $\varphi, \varphi_t, D\varphi , D^2\varphi , A^{\ast}D\varphi$ are uniformly continuous on $(0,T)\times H$;
\item[(ii)] 
$h(t,x)=h_0(t,\|x\|_H)$, where $h_0 \in C^{1,2}((0,T)\times\mathbb R), h_0(t,\cdot)$ is even and for every $t\in(0,T), h_0(t,\cdot)$ is non-decreasing on 
$[0,+\infty)$.
\end{itemize}
\end{definition}

We say that a function is locally bounded if it is bounded on bounded subsets of its domain.
\begin{definition}\label{def:viscosity_solution}
\begin{enumerate}[label=(\roman*)]
\item A locally bounded upper semicontinuous function $v:(0,T]\times H\to\mathbb{R}$ is a viscosity subsolution of \eqref{HJBsemilinear} if $v$ is $B$--upper semicontinuous on $(0,T)\times H$, $v(0,x)\leq g(x)$ for $x\in H$ and, whenever $v-\psi$ has a local maximum at $(t,x) \in (0,T)\times H$ for a test function $\psi =\varphi+h$, then 
\begin{equation*}
\psi_t(t,x) +\langle x, A^{\ast}D\varphi(t,x) \rangle_H + \frac 1 2 \text{Tr} [ \sigma(x) \sigma^{\ast}(x) D^2 \psi(t,x) ] + \mathcal{H}(x,D\psi(t,x)) \geq 0.
\end{equation*}
\item A locally bounded lower semicontinuous function $v:(0,T]\times H\to\mathbb{R}$ is a viscosity subsolution of \eqref{HJBsemilinear} if $v$ is $B$--lower semicontinuous on $(0,T)\times H$, $v(0,x)\geq g(x)$ for $x\in H$ and, whenever $v+\psi$ has a local minimum at $(t,x) \in (0,T)\times H$ for a test function $\psi =\varphi+h$, then 
\begin{equation*}
-\psi_t(t,x) -\langle x, A^{\ast}D\varphi(t,x) \rangle_H + \frac 1 2 \text{Tr} [ \sigma(x) \sigma^{\ast}(x) D^2 \psi(t,x) ] + \mathcal{H}(x,-D\psi(t,x)) \leq 0.
\end{equation*}
\item 
A function $v:(0,T]\times H\to\mathbb{R}$ is a viscosity solution of \eqref{HJBsemilinear} if it is both a viscosity subsolution of \eqref{HJBsemilinear} and a viscosity supersolution of \eqref{HJBsemilinear}.
\end{enumerate}
\end{definition}

We have the following uniqueness result.
\begin{proposition}\label{prop:V_viscosity_sol}
Let Assumptions \ref{assumptionA}, \ref{bsigmalipschitzfirstvariable}, \ref{lglipschitzfirstvariable}, \ref{hp:sigma_independent_control}, \ref{coercivity}  hold.  Then  $V$ is the unique viscosity solution of \eqref{HJBsemilinear} in the set
\begin{equation*}
\begin{split}
S=\Big\{&u:[0,T]\times H\to\mathbb{R}: |u(t,x)|\leq C(1+\|x\|_H^k)\quad\textit{for some}\,\,k\geq 0,
\\
&|u(t,x)-u(t,y)|\leq C\|x-y\|_H\quad\textit{for all}\,\,t\in(0,T],x,y\in H, and
\\
&\lim_{t\to T}|u(t,x)-g(e^{(T-t)A}x)|=0 \quad\textit{uniformly on bounded subsets of}\,\,H\Big\}.
\end{split}
\end{equation*}
Moreover, $V$ is uniformly continuous on every set $[0,\tau]\times B_R$ for every $0<\tau<T, R>0$.
\end{proposition}
Note that we cannot immediately use \cite[Theorem 3.67]{fabbri2017} as in our case the coefficients of the control problem are not bounded in $\Lambda$.
\begin{proof}
\begin{enumerate}
    \item Let $m>0$ and consider the truncated control set  
    $\mathcal{U}^m_t= \left \{a(\cdot)\in \mathcal{U}_t: \| a\|_{L^{\infty}((t,T)\times\Omega;\Lambda)} \leq m  \right\}$.
     Define the value function for the optimal control problem with controls $\mathcal{U}^m_t$ by
     \begin{equation*}
     V^m(t,x) := \inf_{a(\cdot)\in \mathcal{U}^m_t} J(t,x;a(\cdot)).
     \end{equation*}
     By \cite[Theorem 3.67]{fabbri2017}, $V^m$ is the unique viscosity solution of 
     \begin{equation}\label{eq:hjb_V_m}
\begin{cases}
	v_t + \langle Dv,Ax \rangle_H +\frac 1 2\text{Tr} [ \sigma(x) \sigma^{\ast}(x) D^2 v ] + \mathcal{H}_m(x,Dv) =0\\
	v(T,\cdot) = g
\end{cases}
\end{equation}
in $S$. Moreover, it follows from \cite[Proposition 3.62 and Lemma 3.19(i)]{fabbri2017} that $V^m$ is uniformly continuous on $[0,\tau]\times B_R$ for every $0<\tau<T, R>0$.
\item Note that, as in Theorem \ref{th:V_lip}, $V^m(t,\cdot)$ is Lipschitz uniformly in $t$ with Lipschitz constant $R>0$ independent of $m.$ Then, if $\psi$ is a test function in the definition of viscosity solution of \eqref{eq:hjb_V_m} and $V^m-\psi$ (respectively, $V^m+\psi$) achieves a local maximum (respectively, minimum) at $(\overline t,\overline x)$, we have  $\|D \psi(\overline t,\overline x)\|_H \leq R$. Thus, setting $\overline m = m_R$, where $m_R$ is given by Assumption \ref{coercivity}, this implies  
\begin{equation*}
V^m=V^{\overline m}
\end{equation*}
for all $m \geq \overline m$.
\item We prove that 
\begin{equation*}
V^{\overline m}=V.
\end{equation*}
Since obviously $V^{\overline m} \geq V$  as $\mathcal{U}^{\overline m}_t \subset \mathcal{U}_t$, we are left to show that 
\begin{equation}\label{eq:V_m_leq_V}
    V^{\overline m} \leq V.
\end{equation}
Indeed, let $t \in[t,T], x \in H$. Fix $\varepsilon>0$ and let $a^\varepsilon(\cdot)$ be an $\varepsilon$-optimal control for $V$, i.e., $J(t,x,a^\varepsilon(\cdot)) \leq V(t,x)+\varepsilon$. Fix $m > \overline m$ and let $a_m^\varepsilon(\cdot) \in \mathcal{U}^m_t$ be a truncation at $m$ of $a^\varepsilon(\cdot)$, that is, $a_m^\varepsilon(t)=a^\varepsilon(t)$ if $\|a^\varepsilon(t)\|_\Lambda \leq m$ and $a_m^\varepsilon(t)=a_0$ for some fixed $a_0\in\Lambda_0$ otherwise. First, by Step 2 and the definition of $V^m$, we have
\begin{equation}\label{eq:V_bar_m_leq_J_a_m}
    V^{\overline m}(t,x)=V^m(t,x) \leq J(t,x;a_m^\varepsilon(\cdot)).
\end{equation}
We now prove
\begin{equation}\label{eq:J(a_m_eps)_to_J(a_eps)}
   \lim_{m \to \infty} J(t,x;a_m^\varepsilon(\cdot)) =J(t,x;a^\varepsilon(\cdot))
\end{equation}
for all $\varepsilon>0$. Indeed, denoting by $X^\varepsilon_m, X^\varepsilon$ the solutions of the state equation with initial time $t$, initial datum $x$ and controls $a_m^\varepsilon(\cdot)$, $a^\varepsilon(\cdot)$, respectively, we have
\begin{equation*}
\begin{split}
    | J(t,x;a_m^\varepsilon(\cdot)) - J(t,x;a^\varepsilon(\cdot) )| & \leq \int_t^T \mathbb E  |l(X^\varepsilon_m(s),a_m^\varepsilon(s))-l(X^\varepsilon(s),a_m^\varepsilon(s))| \mathrm{d}s\\
    &\quad  + \int_t^T \mathbb E |l(X^{a^\varepsilon}(s),a_m^\varepsilon(s))-l(X^{a^\varepsilon}(s),a^\varepsilon(s))| \mathrm{d}s\\
    &\quad +\mathbb E |g(X^\varepsilon_m(T))-g(X^\varepsilon(T))|=:I_1^m+I_2^m+I_3^m.
\end{split}
\end{equation*}
Noticing that $a_m^\varepsilon(s) \xrightarrow{m \to \infty} a^\varepsilon(s)$ a.s., by dominated convergence we have $I_2^m \xrightarrow{m\to \infty}0.$\\
For $I_1^m,$ by Assumption \ref{lglipschitzfirstvariable}, Lemma \ref{estimatex1x0one} and dominated convergence, we have
\begin{equation*}
\begin{split}
   I_1^m &\leq C\int_t^T \mathbb E \| X^\varepsilon_m(s)-X^\varepsilon(s)\|_H  \mathrm{d}s \leq C \mathbb E \left[ \sup_{s \in [t,T]}\| X^\varepsilon_m(s)-X^\varepsilon(s)\|_H   \right]\\ 
    &\leq C \mathbb E \left[ \sup_{s \in [t,T]}\| X^\varepsilon_m(s)-X^\varepsilon(s)\|_H^2   \right]^{1/2} \leq C \left ( \int_t^T \mathbb E \|a_m^\varepsilon (s)-a^\varepsilon(s) \|_\Lambda^2 \mathrm{d}s \right)^{1/2} \xrightarrow{m \to \infty} 0.
\end{split}
\end{equation*}
A similar argument also shows that $I_3^m   \xrightarrow{m \to \infty} 0$. This proves \eqref{eq:J(a_m_eps)_to_J(a_eps)}. Taking the limit $m \to \infty$ in \eqref{eq:V_bar_m_leq_J_a_m} we therefore obtain
\begin{equation*}
    V^{\overline m}(t,x) \leq J(t,x;a^\varepsilon(\cdot)) \leq V(t,x)+ \varepsilon,
\end{equation*}
where the last inequality follows from the fact that, by definition, $a^\varepsilon(\cdot)$ is an $\varepsilon$-optimal control for $(t,x)$. Letting $\varepsilon \to 0$ we get \eqref{eq:V_m_leq_V}.
\item From the previous steps it follows that $V$ is the unique viscosity solution of \eqref{eq:hjb_V_m} for all $m\geq\overline m$ and, since $V(t,\cdot)$ is Lipschitz, $V$ is also a viscosity solution of \eqref{HJBsemilinear}. If $u\in S$ is another viscosity solution of \eqref{HJBsemilinear}, both $V$ and $u$ are viscosity solutions of \eqref{eq:hjb_V_m} for some $m\geq\overline m$ and by comparison for viscosity solutions of \eqref{eq:hjb_V_m} given by \cite[Theorem 3.54]{fabbri2017} (or existence and uniqueness theorem \cite[Theorem 3.67]{fabbri2017}) we conclude that $u=V$.\qedhere
\end{enumerate}
\end{proof}
Motivated by the regularity results of the previous section, we now impose the following assumption.
\begin{assumption}\label{hp:V_C_11}
 Let for every $0\leq t \leq T$, $V(t,\cdot)\in C^{1,1}(H)$ and let there exist $C\geq 0$ such that $\|DV(t,x)\|_H\leq C, \|DV(t,x)-DV(t,x')\|_H\leq C\|x-x'\|_H$ for all $0\leq t \leq T,x,x'\in H$.
\end{assumption}

For $x,p\in H$, let $\Gamma(x,p) \subset \Lambda_0$ denote the set of all $a^{\ast}\in \Lambda_0$ such that
\begin{equation*}
    \mathcal{H}(x,p,a^{\ast}) = \inf_{a\in \Lambda_0} \mathcal{H}(x,p,a).
\end{equation*}
We impose the following assumption.
\begin{assumption}\label{infimumattained}
 \begin{enumerate}[label=(\roman*)]
\item
   There exists a selection function
    \begin{equation*}
        \gamma : H \times H \to \Lambda_0,\quad 
        (x,p) \mapsto \gamma(x,p) \in \Gamma(x,p),
    \end{equation*}
   which is Lipschitz continuous in both variables.
   \item For every $R>0$ there is a modulus $\omega_R$ such that
   \begin{equation*}
   |l(x,a)-l(x,a')|\leq \omega_R(\|a-a'\|_{\Lambda})
   \end{equation*}
   for all $x\in H$ and $a,a'\in\Lambda_0$ such that $\|x\|_H$, $\|a\|_{\Lambda},\|a'\|_{\Lambda}\leq R$.
   \end{enumerate}
\end{assumption}

For an example in which Assumption \ref{infimumattained} is satisfied, see Example \ref{exampleinfimumattained} below. Due to Assumption \ref{infimumattained}, it is immediate to see that $V$, which is the viscosity solution of the HJB equation \eqref{HJBsemilinear}, is also a viscosity solution of the linear equation
\begin{equation}\label{reducedHJB}
    v_t + \langle Ax,Dv\rangle_H + \frac12 \text{Tr} [ \sigma(x) \sigma^{\ast}(x) D^2v ] + \langle \tilde b(t,x), Dv\rangle_H +\tilde l(t,x) = 0,
\end{equation}
where $\tilde b(t,x):= b(x,\gamma(x,DV(t,x)))$ and $\tilde  l(t,x):=l(x,\gamma(x,DV(t,x)))$.

\begin{theorem}\label{th:optsynth}
Let Assumptions \ref{assumptionA}, \ref{bsigmalipschitz}, \ref{lglipschitzfirstvariable}, \ref{hp:sigma_independent_control}, \ref{coercivity}, \ref{hp:V_C_11}, \ref{infimumattained}  hold.
 Then the pair $(a^{\ast}(s),X^{\ast}(s))$, where
	\begin{equation*}
    \begin{cases}
		a^{\ast}(s) = \gamma(X^{\ast}(s),DV(s,X^{\ast}(s)))\\
        X^{\ast}(s) = X(s;t,x,a^{\ast}(\cdot))
   \end{cases}
	\end{equation*}
	is an optimal couple for
 the optimal control problem \eqref{costfunctional}--\eqref{state} and the control problem has an optimal feedback control. 
\end{theorem}
We point out that this result also applies to deterministic problems, i.e., when $\sigma=0$.
\begin{proof}
Consider the linear equation \eqref{reducedHJB}.
Thanks to Lemma \ref{lemma:DV_locally_uniform_continuous} and by our assumptions, the functions $\tilde b(t,x),\tilde  l(t,x)$ are uniformly continuous on $[0,\tau] \times B_R$ for every $0<\tau<T, R>0$, and moreover 
\begin{equation*}
\|\tilde b(t,x)-\tilde b(t,y)\|_H\leq C\|x-y\|_H
\end{equation*}
for all $t\in[0,T]$, $x,y\in H$, and for every $R>0$ there is a modulus $\sigma_R$ such that
\begin{equation*}
|\tilde l(t,x)-\tilde l(t,y)|\leq \sigma_R(\|x-y\|_H)
\end{equation*}
for all $t\in[0,T]$, $x,y\in H$, $\|x\|_H,\|y\|_H\leq R$. Hence, we can apply \cite[Theorem 3.67]{fabbri2017} (with the control set being a singleton) to get the existence of a unique viscosity solution of \eqref{reducedHJB} in $S$, given by  the Feynman--Kac formula
\begin{equation}\label{eq:feynman_kac}
	v(t,x) = \mathbb{E} \left [ \int_t^T \tilde l(s,X(s)) \mathrm{d}s + g(X(T)) \right ],
\end{equation}
where $X(s)$ is the solution of
\begin{equation}\label{stateexplicitsolution}
\begin{cases}
	\mathrm{d}X(s) = [ A X(s) + \tilde b(s,X(s)) ] \mathrm{d}s + \sigma(X(s))\mathrm{d}W(s)\\
	X(t) = x.
\end{cases}
\end{equation}
We note that in order to apply the comparison theorem \cite[Theorem 3.54]{fabbri2017} (used in \cite[Theorem 3.67]{fabbri2017} to obtain uniqueness of viscosity solutions) to equation \eqref{reducedHJB}, it is enough that the functions $\tilde b$ and $\tilde l$ are uniformly continuous on $[0,\tau] \times B_R$ for every $0<\tau<T, R>0$, rather than on $[0,T] \times B_R$.
Since also $V$ is a viscosity solution of \eqref{reducedHJB}, by uniqueness of viscosity solutions of \eqref{reducedHJB} we have
    \begin{equation}\label{eq:Vrepresentation}
        V(t,x)=v(t,x) = \mathbb{E} \left [ \int_t^T l(X(s),\gamma(X(s),DV(s,X(s)))) \mathrm{d}s + g(X(T)) \right ].
    \end{equation}
This shows the optimality of the feedback control $a^{\ast}(\cdot)$.
\end{proof}
We give an example in which Assumption \ref{infimumattained} is satisfied.
\begin{example}\label{exampleinfimumattained}
   For simplicity we consider $\Lambda_0=\Lambda=H$  but we remark that this assumption could be relaxed. Let $l(x,a) = l_1(x) + l_2(a)$ and $b(x,a) = b(x) - a$, where $l_2\in C^{1,1}(H)$ is such that $l_2(a)-\nu\|a\|_H^2$ is convex for some $\nu>0$. Then
    \begin{equation*}
        \mathcal{F}(x,p,a) := \langle p, b(x) \rangle_H - \langle p,a\rangle_H + l_1(x) + l_2(a).
    \end{equation*}
    In this case, the infimum is attained at
        $a^{\ast} = Dl_2^{-1}(p)$
and Assumption \ref{infimumattained} is satisfied. Indeed, due to the convexity of $l_2(\cdot) - \nu \| \cdot \|_H^2$ and the differentiability of $l_2$, we have
	\begin{equation*}
		l_2(x) \geq l_2(y) + \langle Dl_2(y),x-y\rangle_H  - 2\nu \langle y,x-y\rangle_H
	\end{equation*}
	as well as
	\begin{equation*}
		l_2(y) \geq l_2(x) + \langle Dl_2(x),y-x\rangle_H  - 2\nu \langle x,y-x\rangle_H.
	\end{equation*}
	Adding these two inequalities yields
	\begin{equation}\label{injectivity}
		\langle Dl_2(x) - Dl_2(y),x-y\rangle_H  \geq 2\nu \|x-y\|_H^2.
	\end{equation}
	This implies that $Dl_2$ is injective. Moreover, since $l_2 - \nu \|\cdot \|_H^2$ is convex (Assumption \ref{lgsemiconvex}(i)), $Dl_2 - 2\nu \text{Id}$ is maximal monotone. Therefore, $Dl_2 -2\nu \text{Id} + \lambda \text{Id}$ is surjective for all $\lambda>0$. In particular, $Dl_2$ is surjective and thus bijective. Moreover, by \eqref{injectivity} we have
    \begin{equation*}
    \begin{split}
        \| Dl_2^{-1}(x) - Dl_2^{-1}(y) \|_H^2 &\leq \frac{1}{2\nu} \langle Dl_2^{-1}(x) - Dl_2^{-1}(y), x-y \rangle_H\leq \frac{1}{2\nu} \| Dl_2^{-1}(x) - Dl_2^{-1}(y) \|_H \|x-y\|_H.
    \end{split}
    \end{equation*}
    Therefore, $Dl_2^{-1}$ is Lipschitz continuous with Lipschitz constant $1/(2\nu)$. Thus, Assumption \ref{infimumattained} is satisfied.
\end{example}

\section{Case of Weak $B$-Condition}\label{weakBcase}

In this section, we investigate the regularity of the value function and perform optimal synthesis under the so-called weak $B$-condition (Assumption \ref{assumptionAw}). To do this, we need better regularity properties of the coefficients of the state equation \eqref{stateexplicitsolution}. 

\subsection{Lipschitz Continuity in $\|\cdot\|_{-1}$}\label{lipschitzw}
We will need the following assumptions.
\begin{assumption}\label{bsigmalipschitzfirstvariablew}
\begin{enumerate}[label=(\roman*)]
    \item The function $b$ is continuous on $H\times \Lambda_0$ and there exists a constant $C>0$ such that
    \begin{equation*}
        \|b(x,a)-b(x^{\prime},a) \|_H \leq C \|x-x^{\prime} \|_{-1}
    \end{equation*}
    for all $x,x^{\prime}\in H$ and $a\in\Lambda_0$.
    \item There exists a constant $C>0$ such that
    \begin{equation*}
        \|b(x,a)\|_H \leq C(1+\|x\|_H+\|a\|_{\Lambda} )
    \end{equation*}
    for all $x\in H$ and $a\in\Lambda_0$.
    \item The function $\sigma$ is continuous on $H\times \Lambda_0$ and there exists a constant $C>0$ such that
    \begin{equation*}
        \|\sigma(x,a)-\sigma(x^{\prime},a) \|_{L_2(\Xi,H)} \leq C \|x-x^{\prime} \|_{-1}
    \end{equation*}
    for all $x,x^{\prime}\in H$ and $a\in\Lambda_0$.
    \item There exists a constant $C>0$ such that
    \begin{equation*}
        \|\sigma(x,a)\|_{L_2(\Xi,H)} \leq C(1+\|x\|_H+\|a\|_{\Lambda} )
    \end{equation*}
    for all $x\in H$ and $a\in\Lambda_0$.
\end{enumerate}
\end{assumption}

\begin{assumption}\label{bsigmalipschitzw}
	\begin{enumerate}[label=(\roman*)]
		\item There exists a constant $C>0$ such that
		\begin{equation*}
			\| b(x,a) - b(x^{\prime},a^{\prime} ) \|_H \leq C (\|x-x^{\prime}\|_{-1} + \|a-a^{\prime}\|_{\Lambda} )
		\end{equation*}
		for all $x,x^{\prime}\in H$ and $a,a^{\prime}\in \Lambda_0$.
		\item There exists a constant $C>0$ such that
		\begin{equation*}
			\| \sigma(x,a) - \sigma(x^{\prime},a^{\prime} ) \|_{L_2(\Xi,H)} \leq C (\|x-x^{\prime}\|_{-1} + \|a-a^{\prime}\|_{\Lambda} )
		\end{equation*}
		for all $x,x^{\prime}\in H$ and $a,a^{\prime}\in \Lambda_0$.
	\end{enumerate}
\end{assumption}
We have the following equivalent of Lemma \ref{estimatex1x0one}. 
\begin{lemma}\label{estimatex1x0onew}
	For $x_0,x_1\in H$ and $a_0(\cdot),a_1(\cdot)\in \mathcal{U}_t$, define $X_0(s) = X(s;t,x_0,a_0(\cdot)),X_1(s)=X(s,t,x_1,a_1(\cdot))$. Then it holds:
	\begin{enumerate}[label=(\roman*)]
		\item Let Assumptions \ref{assumptionAw} and \ref{bsigmalipschitzfirstvariablew} be satisfied. Let $a_0(\cdot) = a_1(\cdot) = a(\cdot)\in \mathcal{U}_t$. Then, there are constants $C_1,C_2\geq 0$ independent of $T$, such that
		\begin{equation*}
			\mathbb{E} \left [ \sup_{s\in [t,T]} \| X_1(s) - X_0(s) \|_{-1}^{2} \right ] \leq C_1 \mathrm{e}^{C_2(T-t)}\|x_1 - x_0 \|_{-1}^{2}
		\end{equation*}
        for all $x_0,x_1\in H$ and $a(\cdot)\in \mathcal{U}_t$.
		\item Let Assumptions \ref{assumptionAw} and \ref{bsigmalipschitzw} be satisfied. Then, there are constants $C_1,C_2\geq 0$ independent of $T$, such that
		\begin{equation}\label{ew:-1weak}
			\mathbb{E} \left [ \sup_{s\in [t,T]} \| X_1(s) - X_0(s) \|_{-1}^{2} \right ] \leq C_1 \mathrm{e}^{C_2(T-t)} \left ( \|x_1 - x_0 \|_{-1}^{2} + \mathbb{E} \left [ \int_t^T \| a_1(s) - a_0(s) \|_{\Lambda}^2 \mathrm{d}s \right ] \right )
		\end{equation}
        for all $x_0,x_1\in H$ and let $a_0(\cdot), a_1(\cdot) \in \mathcal{U}_t$.
	\end{enumerate}
\end{lemma}
\begin{proof}
   The proof follows the proof of Lemma \ref{estimatex1x0one} so we only point out the necessary adjustments. We only discuss part $(ii)$. We now apply \cite[Proposition 1.164]{fabbri2017} and use Assumption \ref{assumptionAw} to get
	\begin{equation*}
		\begin{split}
			&\| X_1(s) - X_0(s) \|_{-1}^2\leq \| x_1 - x_0 \|_{-1}^2\\
			&\quad+ 2 \int_t^s (c_0\| X_1(r) - X_0(r) \|_{-1}^2+\langle b(X_1(r),a_1(r)) - b(X_0(r),a_0(r)) , B(X_1(r) - X_0(r)) \rangle_H) 
			\mathrm{d}r\\
			&\quad + \int_t^s \| B^{\frac{1}{2}}(\sigma(X_1(r),a_1(r)) - \sigma(X_0(r),a_0(r))) \|_{L_2(\Xi,H)}^2 \mathrm{d}r\\
			&\quad + 2 \int_t^s \langle B(X_1(r)-X_0(r)), (\sigma(X_1(r),a_1(r)) - \sigma(X_0(r),a_0(r))) \mathrm{d}W(r) \rangle_H.
		\end{split}
	\end{equation*}
	The rest of the proof proceeds as the proof of Lemma \ref{estimatex1x0one} however we now use Assumption \ref{bsigmalipschitzw}.
\end{proof}

Assumption \ref{lglipschitzfirstvariable} is now replaced by the following assumption.
\begin{assumption}\label{lglipschitzfirstvariablew}
\begin{enumerate}[label=(\roman*)]
    \item The function $l$ is continuous on $H\times\Lambda_0$ and there exists a constant $C>0$ such that
    \begin{equation*}
      |l(x,a) - l(x^{\prime},a)| \leq C \|x-x^{\prime}\|_{-1}
    \end{equation*}
    for all $x,x^{\prime}\in H$ and $a\in \Lambda_0$.
     \item There exists a constant $C>0$ such that
    \begin{equation*}
      |l(x,a)| \leq C (1+\|x\|_H+\|a\|_{\Lambda}^2)
    \end{equation*}
    for all $x\in H$ and $a\in \Lambda_0$.
    \item There exists a constant $C>0$ such that
    \begin{equation*}
        |g(x) - g(x^{\prime})| \leq C \|x-x^{\prime}\|_{-1}
    \end{equation*}
    for all $x,x^{\prime}\in H$.
    \item If $\Lambda_0$ is unbounded we assume that $l$ and $g$ are bounded from below.
\end{enumerate}
\end{assumption}

The value function now is Lipschitz continuous in $x$ in the $\|\cdot\|_{-1}$ norm:

\begin{theorem}\label{th:V_lipw}
    Let Assumptions \ref{assumptionAw}, \ref{bsigmalipschitzfirstvariablew} and \ref{lglipschitzfirstvariablew} be satisfied. Then there are constants $C_1,C_2$ independent of $T$, such that
	\begin{equation*}
		|V(t,x) - V(t,y)| \leq C_1 \mathrm{e}^{C_2(T-t)} \|x-y\|_{-1}
	\end{equation*}
	for all $t\in [0,T]$ and $x,y\in H$.
\end{theorem}

The proof of Theorem \ref{th:V_lipw} is the same as the proof of Theorem \ref{th:V_lip} if one uses \eqref{ew:-1weak} so it is omitted.

\subsection{Semiconcavity in $H_{-1}$}\label{semiconcavityw}

We now show that $V$ extends to a function which is semiconcave in $H_{-1}$. We need to strengthen Assumptions \ref{bsigmafrechetfirstvariable} and \ref{bsigmafrechet}. We point out that Assumption \ref{bsigmafrechetw} is not needed in this subsection.
\begin{assumption}\label{bsigmafrechetfirstvariablew}
\begin{enumerate}[label=(\roman*)]
    \item Let $b:H\times\Lambda \to H$ be Fr\'echet differentiable in the first variable and let there be a constant $C>0$ such that
    \begin{equation*}
        \| (D_xb(x,a) - D_xb(x^{\prime},a))(x-x^{\prime}) \|_{-1} \leq C \|x-x^{\prime}\|_{-1}
    \end{equation*}
    for all $x,x^{\prime}\in H$ and $a\in \Lambda_0$.
    \item Let $\sigma:H\times\Lambda \to L_2(\Xi,H)$ be Fr\'echet differentiable in the first variable and let there be a constant $C>0$ such that
    \begin{equation*}
        \| (D_x\sigma(x,a) - D_x\sigma(x^{\prime},a))(x-x^{\prime}) \|_{L_2(\Xi,H_{-1})} \leq C \|x-x^{\prime}\|_{-1}
    \end{equation*}
    for all $x,x^{\prime}\in H$ and $a\in \Lambda_0$.
\end{enumerate}
\end{assumption}

\begin{assumption}\label{bsigmafrechetw}
	\begin{enumerate}[label=(\roman*)]
		\item Let $b:H\times\Lambda\to H$ be Fr\'echet differentiable and let there be a constant $C>0$ such that
        \begin{equation*}
        \begin{split}
            &\| (D_x b(x,a) - D_x b(x^{\prime},a^{\prime}))(x-x^{\prime}) + ( D_a b(x,a) - D_a b(x^{\prime},a^{\prime}))(a-a^{\prime})\|_{-1}\\
            &\leq C \left ( \|x-x^{\prime}\|^2_{-1} + \| a-a^{\prime} \|_{\Lambda}^2 \right )
        \end{split}
        \end{equation*}
        for all $x,x^{\prime}\in H$, $a,a^{\prime}\in \Lambda_0$.
		\item Let $\sigma:H\to L_2(\Xi,H)$ be Fr\'echet differentiable and let there be a constant $C>0$ such that
		\begin{equation*}
			\| (D\sigma(x) - D\sigma(x^{\prime}))(x-x^{\prime})\|_{L_2(\Xi,H_{-1})} \leq C \|x-x^{\prime}\|_{-1}^2
		\end{equation*}
		for all $x,x^{\prime}\in H$.
	\end{enumerate}
\end{assumption}

For given $x_0,x_1\in H$ and $a_0(\cdot), a_1(\cdot)\in \mathcal{U}_t$, let $X_0$ and $X_1$ be given by \eqref{x1x0differentcontrols}, and for $\lambda\in [0,1]$ we define $a_{\lambda}(\cdot), x_{\lambda}, X_{\lambda}, X^{\lambda}$ as in
\eqref{lambdadefinition}. We now have the following version of Lemma \ref{estimatexlambdasamecontrol}:

\begin{lemma}\label{estimatexlambdasamecontrolw}
\begin{enumerate}[label=(\roman*)]
	\item Let Assumptions \ref{assumptionAw}, \ref{bsigmalipschitzfirstvariablew} and \ref{bsigmafrechetfirstvariablew} be satisfied. Let $a_0(\cdot) = a_1(\cdot) = a(\cdot)\in \mathcal{U}_t$. There are constants $C_1,C_2\geq 0$ independent of $T$ where $C_2$ depends only on the constant $C_2$ in Lemma \ref{estimatex1x0onew} and $c_0$ in Assumption \ref{assumptionAw}, such that
	\begin{equation*}
		\mathbb{E} \left [ \sup_{s\in [t,T]} \| X^{\lambda}(s) - X_{\lambda}(s) \|_{-1} \right ] \leq C_1 \mathrm{e}^{C_2(T-t)} \lambda (1-\lambda) \|x_1-x_0\|_{-1}^2
	\end{equation*}
    for all $\lambda \in [0,1]$, $x_0,x_1\in H$ and $a(\cdot)\in \mathcal{U}_t$.
	\item Let Assumptions \ref{assumptionAw}, \ref{bsigmalipschitzw} and \ref{bsigmafrechetw} be satisfied. There are constants $C_1,C_2\geq 0$, where $C_2$ depends only on the constant $C_2$ in Lemma \ref{estimatex1x0onew} and $c_0$ in Assumption \ref{assumptionAw}, such that
	\begin{equation*}
			\mathbb{E} \left [ \sup_{s\in [t,T]} \| X^{\lambda}(s) - X_{\lambda}(s) \|_{-1} \right ] \leq C_1 \mathrm{e}^{C_2(T-t)} \lambda (1-\lambda) \left ( \|x_1-x_0\|_{-1}^2 + \mathbb{E} \left [ \int_t^T \|a_0(s) - a_1(s) \|_{\Lambda}^2 \mathrm{d}s \right ] \right )
	\end{equation*}
    for all $\lambda \in [0,1]$, $x_0,x_1\in H$ and $a_0(\cdot),a_1(\cdot) \in \mathcal{U}_t$.
\end{enumerate}
\end{lemma}

\begin{proof}
The proof is similar to the proof of Lemma \ref{estimatexlambdasamecontrol} with little modifications. Repeating its steps and using Assumption \ref{bsigmafrechetw}, instead of \eqref{estimateb} we now have
\begin{equation*}
	\begin{split}
		&\| \lambda b(X_1(s),a_1(s)) +(1-\lambda) b(X_0(s),a_0(s)) - b(X^{\lambda}(s),a_{\lambda}(s)) \|_{-1}\\
		&\leq C \lambda (1-\lambda) \left ( \|X_1(s) - X_0(s)\|_{-1}^2 + \|a_1(s) -a_0(s)\|_{\Lambda}^2 \right )
	\end{split}
\end{equation*}
and in place of \eqref{estimatesigma},
\begin{equation*}
	\| \lambda \sigma(X_1(s)) +(1-\lambda) \sigma(X_0(s)) - \sigma(X^{\lambda}(s)) \|_{L_2(\Xi,H_{-1})} 
	\leq C \lambda (1-\lambda) \|X_1(s) - X_0(s)\|_{-1}^2.
\end{equation*}
Then, applying \cite[Proposition 1.164]{fabbri2017} and Assumption \ref{assumptionAw}, we obtain in place of \eqref{itosformula4}
\begin{equation}\label{itosformula4w}
	\begin{split}
		&\| X^{\lambda}(s) - X_{\lambda}(s) \|_{-1}^2\leq 2c_0 \int_t^s \| X^{\lambda}(r) - X_{\lambda}(r) \|_{-1}^2 \mathrm{d}r\\
		&+ 2 \int_t^s \langle \lambda b(X_1(r),a_1(r)) +(1-\lambda) b(X_0(r),a_0(r)) - b(X^{\lambda}(r),a_\lambda(r)) , X^{\lambda}(r) - X_{\lambda}(r) \rangle_{-1} \mathrm{d}r\\
        &\quad +2 \int_t^s \langle b(X^{\lambda}(r),a_{\lambda}(r)) - b(X_{\lambda}(r),a_{\lambda}(r)), X^{\lambda}(r) - X_{\lambda}(r) \rangle_{-1} \mathrm{d}r\\
		&\quad + 2 \int_t^s \| ( \lambda \sigma(X_1(r)) +(1-\lambda) \sigma(X_0(r)) - \sigma(X^{\lambda}(r))) \|_{L_2(\Xi,H_{-1})}^2 \mathrm{d}r\\
        &\quad +2\int_t^s \| \sigma(X^{\lambda}(r) - \sigma(X_{\lambda}(r) \|_{L_2(\Xi,H_{-1})}^2 \mathrm{d}r\\
		&\quad + 2 \int_t^s \left \langle X^{\lambda}(r) - X_{\lambda}(r), ( \lambda \sigma(X_1(r)) +(1-\lambda) \sigma(X_0(r)) - \sigma(X^{\lambda}(r)) ) \mathrm{d}W(r)\right \rangle_{-1}\\
        &\quad + 2 \int_t^s \left \langle X^{\lambda}(r) - X_{\lambda}(r), ( \sigma(X^{\lambda}(r)) - \sigma(X_{\lambda}(r)) ) \mathrm{d}W(r)\right \rangle_{-1}.
	\end{split}
\end{equation}
Now, the proof follows along the same lines as the proof of Lemma \ref{estimatexlambdasamecontrol}.
\end{proof}

Finally, we strengthen Assumption \ref{lgsemiconcave} to the following one.
\begin{assumption}\label{lgsemiconcavew}
    \begin{enumerate}[label=(\roman*)]
        \item Let $l(\cdot,a)$ be semiconcave in $H_{-1}$ uniformly in $a\in \Lambda_0$.
        \item Let $g$ be semiconcave in $H_{-1}$.
    \end{enumerate}
\end{assumption}

\begin{theorem}\label{th:semiconcw}
    Let Assumptions \ref{assumptionAw}, \ref{bsigmalipschitzfirstvariablew}, \ref{lglipschitzfirstvariablew}, \ref{bsigmafrechetfirstvariablew} and \ref{lgsemiconcavew} be satisfied. Then, for every $t\in[0,T]$, the function $V(t,\cdot)$ is semiconcave in $H_{-1}$ with semiconcavity constant $C_1 \mathrm{e}^{C_2(T-t)}$ for some $C_1,C_2\geq 0$ independent of $t,T$.
\end{theorem}
\begin{proof}
The proof repeats the steps of the proof of Theorem \ref{th:semiconc}, now using the new assumptions, and Lemma \ref{estimatex1x0onew}(i) and Lemma \ref{estimatexlambdasamecontrolw}(i) instead of Lemma \ref{estimatex1x0one}(i) and Lemma \ref{estimatexlambdasamecontrol}(i).
\end{proof}

\subsection{Semiconvexity}\label{semiconvexityw}

In this subsection, we show some cases when the value function is semiconvex. Since the classical Nemytskii operator does not satisfy Assumption \ref{bsigmalipschitzfirstvariablew}, we will only consider versions of Cases 1 and 2 from Section \ref{semiconvexity}. However, see Remark \ref{rem:case3_delay} for Case 3 in the context of delay equations.

\subsubsection{Uniform Convexity of Running Cost in The Control Variable}

Assumption \ref{lgsemiconvex} is now changed to Assumption \ref{lgsemiconvexw}  below.
\begin{assumption}\label{lgsemiconvexw}
	\begin{enumerate}[label=(\roman*)]
		\item There exist constants $C,\nu \geq 0$ such that the map
		\begin{equation*}
			H\times\Lambda_0 \ni (x,a)\mapsto l(x,a) + C\|x\|_{-1}^2 - \nu \|a\|^2_{\Lambda}
		\end{equation*}
		is convex.
    \item Let $g:H\to\mathbb{R}$ be semiconvex in $H_{-1}$.
	\end{enumerate}
\end{assumption}

\begin{theorem}\label{convexity1w}
	Let Assumptions \ref{assumptionAw}, \ref{bsigmalipschitzw}, \ref{lglipschitzfirstvariablew}, \ref{bsigmafrechetw} and \ref{lgsemiconvexw} be satisfied. Then there exists a constant $\nu_0$ depending only on the data of the problem such that if $\nu\geq \nu_0$, then $V(t,\cdot)$ is semiconvex in $H_{-1}$ with constant $C_1\mathrm{e}^{C_2T}$.
\end{theorem}
\begin{proof}
The proof is analogous to the proof of Theorem \ref{convexity1} with obvious changes due to the new assumptions, and using Lemma \ref{estimatex1x0onew}(ii) and Lemma \ref{estimatexlambdasamecontrolw}(ii) instead of Lemma \ref{estimatex1x0one}(ii) and Lemma \ref{estimatexlambdasamecontrol}(ii).
\end{proof}

\subsubsection{Linear State Equation and Convex Costs}

In this case we need to enhance Assumption \ref{blinearlgconvex} to the following one.
\begin{assumption}\label{blinearlgconvexw}
	\begin{enumerate}[label=(\roman*)]
		\item The functions $b,\sigma$ extend to affine and continuous functions $b:H_{-1}\times\Lambda\to H$ and $\sigma: H_{-1} \times \Lambda \to L_2(\Xi,H)$.
		\item The functions $l:H\times\Lambda_0\to\mathbb{R}$ and $g:H\to\mathbb{R}$ are convex.
	\end{enumerate}
\end{assumption}

The proof of Theorem \ref{convexity2w} is the same as that of Theorem \ref{convexity2}.
\begin{theorem}\label{convexity2w}
	Let Assumptions \ref{assumptionAw}, \ref{lglipschitzfirstvariablew} and  \ref{blinearlgconvexw} be satisfied. Then, for every $t\in [0,T]$, the function $V(t,\cdot)$ is convex.
\end{theorem}

\subsection{$C^{1,1}$ Regularity of the Value Function}
Similarly to Section \ref{sec:C11reg}, since $V(t,\cdot), 0\leq t\leq T$, is Lipschitz in the $\|\cdot\|_{-1}$ norm, if it is semiconvex in $H_{-1}$ and semiconcave in $H_{-1}$, $V$ can be extended to a function (still denoted by $V$) such that $V(t,\cdot)\in C^{1,1}(H_{-1})$.

\subsection{Optimal Synthesis}\label{sec:optsyntw}

In the weak $B$-condition case, the following equivalent of the uniqueness result Proposition \ref{prop:V_viscosity_sol} is in fact easier to prove than in the strong $B$-condition case.

\begin{proposition}\label{prop:V_viscosity_solw}
Let Assumptions \ref{assumptionAw}, \ref{hp:sigma_independent_control}, \ref{coercivity}, \ref{bsigmalipschitzfirstvariablew}, \ref{lglipschitzfirstvariablew},  hold.  Then  $V$ is the unique viscosity solution of \eqref{HJBsemilinear} in the set
\begin{equation*}
\begin{split}
S=\Big\{&u:[0,T]\times H\to\mathbb{R}: |u(t,x)|\leq C(1+\|x\|_H^k)\quad\textit{for some}\,\,k\geq 0,
\\
&|u(t,x)-u(t,y)|\leq C\|x-y\|_H\quad\textit{for all}\,\,t\in(0,T],x,y\in H, and
\\
&\lim_{t\to T}|u(t,x)-g(x)|=0 \quad\textit{uniformly on bounded subsets of}\,\,H\Big\}.
\end{split}
\end{equation*}
Moreover, $V$ is uniformly continuous in the $|\cdot|\times\|\cdot\|_{-1}$ norm on every set $[0,T]\times B_R$ for every $R>0$.
\end{proposition}

\begin{proof}
The proof repeats the steps of the proof of Proposition \ref{prop:V_viscosity_sol} with a few adjustments. Using \cite[Theorem 3.66]{fabbri2017}, the functions $V^m$ are the unique viscosity solutions of \eqref{eq:hjb_V_m}
in $S$ and it follows from \cite[Proposition 3.61]{fabbri2017} that $V^m$ is uniformly continuous in the $|\cdot|\times\|\cdot\|_{-1}$ norm on every set $[0,T]\times B_R$ for every $R>0$. We then show that $V^m=V^{\bar m}$ for $m\geq \bar m$ for some $\bar m$ as in Step 2, and then show that $V^m$ converge to $V$ as in Step 3. The only difference is that we now use Assumption \ref{lglipschitzfirstvariablew} and Lemma \ref{estimatex1x0onew}. We conclude as in Step 4 noting that uniqueness now follows from \cite[Theorem 3.50]{fabbri2017} (or existence and uniqueness theorem \cite[Theorem 3.66]{fabbri2017}).
\end{proof}

We strengthen Assumption \ref{hp:V_C_11}.
\begin{assumption}\label{hp:V_C_11w}
 Let for every $0\leq t \leq T$, $V(t,\cdot)\in C^{1,1}(H)$ and let there exist $C\geq 0$ such that $\|DV(t,x)\|_H\leq C, \|DV(t,x)-DV(t,x')\|_H\leq C\|x-x'\|_{-1}$ for all $0\leq t \leq T,x,x'\in H$.
\end{assumption}

Lemma \ref{lemma:DV_locally_uniform_continuous} gives us now that $DV$ is uniformly continuous on $[0,T]\times B_R$ for every $R>0$. We change Assumption \ref{infimumattained} to
\begin{assumption}\label{infimumattainedw}
 \begin{enumerate}[label=(\roman*)]
\item
   There exists a selection function
    \begin{equation*}
        \gamma : H \times H \to \Lambda_0,\quad
        (x,p) \mapsto \gamma(x,p) \in \Gamma(x,p),
    \end{equation*}
   which is Lipschitz continuous in both variables with respect to the norm $\|\cdot\|_{-1}\times\|\cdot\|_H$.
    \item For every $R>0$ there is a modulus $\omega_R$ such that
   \begin{equation*}
   |l(x,a)-l(x,a')|\leq \omega_R(\|a-a'\|_{\Lambda})
   \end{equation*}
   for all $x\in H$ and $a,a'\in\Lambda_0$ such that $\|x\|_H$, $\|a\|_{\Lambda}$, $\|a'\|_{\Lambda}\leq R$.
   \end{enumerate}
\end{assumption}
\begin{theorem}\label{th:optsynthw}
Let Assumptions \ref{assumptionAw}, \ref{bsigmalipschitzw}, \ref{lglipschitzfirstvariablew}, \ref{hp:sigma_independent_control}, \ref{coercivity}, \ref{hp:V_C_11w}, \ref{infimumattainedw} hold. Then the pair $(a^{\ast}(s),X^{\ast}(s))$, where
	\begin{equation*}
    \begin{cases}
		a^{\ast}(s) = \gamma(X^{\ast}(s),DV(s,X^{\ast}(s))\\
        X^{\ast}(s) = X(s;t,x,a^{\ast}(\cdot))
    \end{cases}
	\end{equation*}
	is an optimal couple for the optimal control problem \eqref{costfunctional}--\eqref{state} and the control problem has an optimal feedback control. 
\end{theorem}
\begin{proof}
We just need to update the proof of Theorem \ref{th:optsynth}. We notice that by our assumptions, the functions $\tilde b(t,x)= b(x,\gamma(x,DV(t,x)))$ and $\tilde  l(t,x)=l(x,\gamma(x,DV(t,x)))$ are uniformly continuous on bounded subsets of $[0,T]\times H$, and moreover 
\begin{equation*}
\|\tilde b(t,x)-\tilde b(t,y)\|_H\leq C\|x-y\|_{-1}
\end{equation*}
for all $t\in[0,T]$, $x,y\in H$, and for every $R>0$ there is a modulus $\sigma_R$ such that
\begin{equation*}
\|\tilde l(t,x)-\tilde l(t,y)\|_H\leq \sigma_R(\|x-y\|_{-1})
\end{equation*}
for all $t\in[0,T]$ and $x,y\in H$ such that $\|x\|_H,\|y\|_H\leq R$. Thus, we can apply  \cite[Theorem 3.66]{fabbri2017} to claim that \eqref{reducedHJB} has a unique viscosity solution in $S$ given by \eqref{eq:feynman_kac}. However, since $V$ is also a viscosity solution of \eqref{reducedHJB}, using uniqueness we obtain \eqref{eq:Vrepresentation}, which gives the optimality of the feedback control $a^{\ast}(\cdot)$.
\end{proof}

\section{Applications}\label{section:applications}

In this section we apply our theory to controlled stochastic partial differential equations as well as controlled stochastic delay differential equations.

\subsection{Controlled Stochastic Reaction-Diffusion Equations}\label{example_reaction_diffusion}
In this first example, we consider the state equation \eqref{state} for $H=L^2(\mathcal{O})$, where $\mathcal{O}\subset \mathbb{R}^d$, $d\leq 4$, is a bounded domain with a sufficiently regular boundary. Let
\begin{equation*}
\begin{cases}
    A := \sum_{i,j=1}^d \partial_i ( a_{ij} \partial_j) + \sum_{i=1}^d b_i \partial_i + c\\
    \mathcal{D}(A) = H_0^1(\mathcal{O}) \cap H^2(\mathcal{O})
\end{cases}
\end{equation*}
where $a_{ij}= a_{ji}, b_i\in C^1(\bar{\mathcal{O}})$, $i,j=1,\dots,d$, $c \in C(\bar{\mathcal{O}})$, and there is a constant $\theta >0$ such that
\begin{equation*}
    \sum_{i,j=1}^d a_{ij} \xi_i \xi_j \geq \theta | \xi |^2
\end{equation*}
for all $\xi\in \mathbb{R}^d$. Under these assumptions, $A$ generates a positivity preserving $C_0$-semigroup, see \cite[Theorem 4.2]{ouhabaz2005}. Thus, Assumption \ref{comparisonprinciple}(i) is satisfied. Furthermore, for an appropriate choice of $B$, the strong $B$-condition (Assumption \ref{assumptionA}) is satisfied, see \cite[Example 3.16]{fabbri2017}. Moreover, $A$ satisfies G\r{a}rding's inequality
\begin{equation}\label{garding}
    \langle Ax,x \rangle_{H^{-1}(\mathcal{O}) \times H^1_0(\mathcal{O})} \leq - C_1 \|x\|_{H^1_0(\mathcal{O})}^2 + C_2 \|x\|^2_H
\end{equation}
for some constants $C_1,C_2 > 0$, see \cite[Example 11.5]{renardy_rogers_2004}. Note that the operator $A$ is not dissipative in general. However, one can consider the operator $\tilde{A} := A- C_2 I$, which is dissipative due to \eqref{garding} and then modify the bounded term $\mathfrak{b}$ that is introduced below.

Let $\mathfrak{b}:\mathbb{R} \to \mathbb{R}$ model some local reaction function, let $\Lambda=L^2(\mathcal{O})$, and define $b(x,a)(\xi) := \mathfrak{b}(x(\xi)) - a(\xi)$ for $x,a\in L^2(\mathcal{O})$. Furthermore, let us consider additive noise, i.e., $\sigma(x,a) \equiv \sigma \in L_2(\Xi,H)$. In this case, the state equation reduces to
\begin{equation*}
\begin{cases}
	\mathrm{d}X(s) = [ A X(s) + \mathfrak{b}(X(s)) - a(s) ] \mathrm{d}s + \sigma \mathrm{d}W(s),\quad s\in [t,T]\\
	X(t) = x\in L^2(\mathcal{O}).
\end{cases}
\end{equation*}
Control problems associated with such state equations arise in many applications, see \cite[Section 2.6.1]{fabbri2017} and the references therein.

If $\mathfrak{b}$ is differentiable and its derivative is bounded, $b$ satisfies Assumption \ref{bsigmalipschitz}(i). The Gateaux derivative of $b$ with respect to $x$ is given by
\begin{equation*}
    D_xb(x,a)(\xi) = \mathfrak{b}^{\prime}(x(\xi)),
\end{equation*}
where the right-hand side acts as a multiplication operator in $L^2(\mathcal{O})$. We note that even if $\mathfrak{b}\in C^{1,1}(\mathbb{R})$, the associated Nemytskii operator is not in $C^{1,1}(H)$. However, this regularity is used in the proof of Lemma \ref{estimatexlambdasamecontrol}(i), which in turn is used to prove the semiconcavity in Theorem \ref{th:semiconc}. Therefore, we need to modify the proofs in Subsection \ref{semiconcavity}.

In the present example, using variational methods, G\r{a}rdings inequality \eqref{garding} can be used to prove the following improved version of Lemma \ref{estimatex1x0one}(i):
\begin{lemma}\label{lemma_reaction_diffusion}
    Let Assumption \ref{bsigmalipschitzfirstvariable} be satisfied. Let $a_0(\cdot) = a_1(\cdot) = a(\cdot)\in \mathcal{U}_t$. Then, there is a constant $C\geq 0$, such that
		\begin{equation*}
			\mathbb{E} \left [ \sup_{s\in [t,T]} \| X_1(s) - X_0(s) \|_H^{2} + \int_t^T \| X_1(s) - X_0(s)\|_{H^1_0(\mathcal{O})}^2 \mathrm{d}s \right ] \leq C \|x_1 - x_0 \|_H^{2}
		\end{equation*}
        for all $x_0,x_1\in H$ and $a(\cdot)\in \mathcal{U}_t$.
\end{lemma}
The proof follows along the same lines as the proof of Lemma \ref{estimatex1x0one} but uses the variational It\^o formula, see \cite[Theorem 4.2.5]{liu_roeckner_2015}, and equation \eqref{garding}.

Using this lemma and assuming that $\mathfrak{b}\in C^{1,1}(\mathbb{R})$, we can still prove Lemma \ref{estimatexlambdasamecontrol}(i) even though the Nemystkii operator is not in $C^{1,1}(H)$. Indeed, instead of equation \eqref{estimateb1}, we now have
\begin{equation*}
\begin{split}
	&\| \lambda b(X_1(s)) +(1-\lambda) b(X_0(s)) - b(X^{\lambda}(s)) \|_H\\
    &\leq \lambda (1-\lambda) \left ( \int_{\mathcal{O}} \int_0^1 ( \mathfrak{b}^{\prime}(\bar X_1(\theta,\xi)) - \mathfrak{b}^{\prime}(\bar X_0(\theta,\xi)))^2 (X_1(s,\xi) - X_0(s,\xi))^2 \mathrm{d}\theta \mathrm{d}\xi \right )^{\frac12}\\
    &\leq C \lambda (1-\lambda) \| (X_1(s) - X_0(s)) \|_{L^4(\mathcal{O})}^2 \leq C \lambda (1-\lambda) \| (X_1(s) - X_0(s)) \|_{H^1_0(\mathcal{O})}^2,
 \end{split}
\end{equation*}
where in the last step we used the fact that $d\leq 4$ and therefore $H^1_0(\mathcal{O})$ is continuously embedded in $L^4(\mathcal{O})$, due to the Sobolev embedding theorem. Thus, using similar arguments as in equation \eqref{estimatingb} and using the previous estimate, we obtain
\begin{equation*}
\begin{split}
    & \int_t^T \langle \lambda b(X_1(s)) +(1-\lambda) b(X_0(s)) - b(X^{\lambda}(s)) , X^{\lambda}(s) - X_{\lambda}(s) \rangle_H \mathrm{d}s \\
    &\leq \frac{1}{16} \sup_{s\in [t,T]} \| X^{\lambda}(s) - X_{\lambda}(s) \|^2_H + C \lambda^2 (1-\lambda)^2 \left ( \int_t^T \| (X_1(s) - X_0(s)) \|_{H^1_0(\mathcal{O})}^2 \mathrm{d}s \right )^2.
\end{split}
\end{equation*}
Now, the remainder of the proof of Lemma \ref{estimatexlambdasamecontrol}(i) follows along the same lines, except that we use Lemma \ref{lemma_reaction_diffusion} instead of Lemma \ref{estimatex1x0one} to estimate the last term in the previous inequality. Therefore, the proof of Theorem \ref{th:semiconc} still holds under appropriate assumptions on the coefficients of the cost functional (which are discussed below).

Note that Assumptions \ref{bsigmalipschitz}(ii), \ref{bsigmafrechet}(ii), and \ref{hp:sigma_independent_control} are trivially satisfied for additive noise.

For the cost functional, let us assume that the running costs are separated as in Example \ref{exampleinfimumattained} and are of Nemytskii type, i.e.,
\begin{equation}
	J(t,x;a(\cdot)) := \mathbb{E} \left [ \int_t^T \int_{\mathcal{O}}\left( l_1(X(s,\xi)) + l_2(a(s,\xi))\right) \mathrm{d}\xi \mathrm{d}s + \int_{\mathcal{O}} g(X(T,\xi)) \mathrm{d}\xi \right ]
\end{equation}
for some functions $l_1,l_2,g:\mathbb{R}\to \mathbb{R}$. Assuming that $l_1$ and $g$ are Lipschitz continuous, $l_2$ is quadratically bounded and $l_1,l_2$ and $g$ are bounded below, Assumption \ref{lglipschitzfirstvariable} is satisfied. Moreover, if $l_1$ and $g$ are semiconcave, Assumption \ref{lgsemiconcave} is satisfied. Furthermore, if $l_1,l_2$ and $g$ are convex, and $l_1$ and $g$ are non-increasing, Assumption \ref{comparisonprinciple2} is satisfied. Finally, if $\mathfrak{b}$ is concave, Assumption \ref{comparisonprinciple3} is satisfied. Thus, we can apply Theorem \ref{th:V_lip} to prove Lipschitz continuity of the value function, Theorem \ref{th:semiconc} to obtain semiconcavity (where we need to make the modifications outlined above), and Theorem \ref{th:convcomparison1} to obtain convexity of the value function. Therefore, the value function is $C^{1,1}$, i.e., Assumption \ref{hp:V_C_11} is satisfied. Assuming $l_2(a)/|a| \to +\infty$ as $|a|\to \infty$ implies that Assumption \ref{coercivity} is satisfied, see Remark \ref{rem:H_attains_inf_bounded_set_H}. Assuming additionally that $l_2(a)-\nu |a|^2$ is convex for some $\nu >0$, we can combine this with Example \ref{exampleinfimumattained} which shows that Assumption \ref{infimumattained} is satisfied. Thus, we can apply Theorem \ref{th:optsynth} to obtain the optimal control in the feedback form
\begin{equation*}
    a^{\ast}(s) = Dl_2^{-1}(DV(s,X^{\ast}(s))).
\end{equation*}

\subsection{Controlled Stochastic Delay Differential Equations}
In this second example, we apply our theory to optimal control problems for stochastic delay differential equations (SDDEs). These problems are a natural generalization of stochastic control problems for Markovian SDEs in $\mathbb R^n$ and arise in many applied sciences such as economics and finance. We refer to  \cite{deFeo-Federico-Swiech} for an extensive introduction to the subject, see also \cite[Section 2.6.8]{fabbri2017}.

Let us consider the following controlled SDDE
\begin{equation}
\label{eq:SDDE}
\begin{cases}
\mathrm{d}y(s) =  b_0 \left ( y(s),\int_{-d}^0 \eta_1(\xi)y(s+\xi)\mathrm{d}\xi ,a(s) \right)
         \mathrm{d}s 
  + \sigma_0 \left (y(s),\int_{-d}^0 \eta_2(\xi)y(s+\xi)\mathrm{d}\xi , a(s) \right) \mathrm{d}W(s),\\
y(t)=x_0, \quad y(t+\xi)=x_1(\xi) \quad \forall \xi\in[-d,0),
\end{cases}
\end{equation}
where  $d>0$ is the maximum delay, $s \geq t \geq 0,$  
 $x_0 \in \mathbb{R}^n$ and $x_1 \in L^2([-d,0]; \mathbb{R}^n)$ are the initial conditions, $b_0 \colon \mathbb{R}^n \times \mathbb{R}^h \times \Lambda_0 \to \mathbb{R}^n$, $\sigma_0 \colon \mathbb{R}^n \times \mathbb{R}^h \times \Lambda_0 \to \mathbb R^{n\times q}$
 and $W(t)$ is a standard Wiener process on $[t,\infty)$ taking values in $\mathbb R^q$. Here, the control set  $\Lambda_0 \subset \Lambda:= \mathbb  R^n$ is convex,
 $\eta_{i}:[-d,0]\to \mathbb  R^{h \times n}$ for $i=1,2$ are the delay  kernels and if $\eta_{i}^{j}$ is the $j$-th row of $\eta_i(\cdot)$, for $j=1,\dots,h$, then  $\eta_i^{j}$ is a function in the  Sobolev space $ W^{1,2}([-d,0];\mathbb R^n)$  and $\eta_i^{j}(-d)=0$. The running cost is $l_0 \colon \mathbb{R}^n \times \Lambda_0 \to \mathbb{R}$ and the terminal cost is  $g_0 \colon \mathbb{R}^n  \to \mathbb{R}$. Given an initial state $(x_0,x_1) \in \mathbb R^n \times L^2([-d,0]; \mathbb{R}^n),$ the  functional to be minimized is
\begin{equation*}
J(t,x_0,x_1;a(\cdot)) :=
\mathbb E\left[\int_t^{T} 
l_0(y(s),a(s)) \mathrm{d}s + g_0(y(T)) \right],
\end{equation*}
over all admissible controls $a(\cdot)$.
In this setting and under the assumptions of the next subsection, there exists a unique strong solution to \eqref{eq:SDDE}.
\begin{example}
Examples of these types of problems with delays are for instance Merton portfolio problems (see e.g.  \cite{deFeo-Federico-Swiech}), optimal advertising problems (see e.g. \cite{gozzi_marinelli_2004}, \cite{deFeo-Federico-Swiech}), optimization of pension funds (see e.g. \cite{federico_2011})
and optimal control of endogenous growth models with time-to-build (see e.g. \cite{bambi_fabbri_gozzi (2012)}, \cite{goldys_2} for the deterministic case). 
\end{example}

Set $H:=\mathbb{R}^n \times L^2([-d,0]; \mathbb{R}^n)$. We use the notation $x=(x_0,x_1) \in H$ or $x=[x_0,x_1]^T$. Defining $A \colon D(A) \subset H \to H$, $ b \colon H \times \Lambda_0 \to H$  and $\sigma \colon H \times \Lambda_0 \to L(\mathbb{R}^q,H)$, respectively, by  
\begin{align*}
& A x:= \begin{bmatrix}
-x_0\\
x_1'
\end{bmatrix}, \quad D(A):= \left\{ x=(x_0,x_1) \in H: x_1 \in W^{1,2}([-d,0];\mathbb R^n), \ x_1(0)=x_0\right\},\\
& b(x,a) :=
\begin{bmatrix}
b_0 \left ( x_0,\int_{-d}^0 \eta_1(\xi)x_1(\xi)\mathrm{d}\xi ,a \right)+x_0\\
0
\end{bmatrix},\quad  \sigma(x,u)w:=\begin{bmatrix}
\sigma_0 \left ( x_0,\int_{-d}^0 \eta_2(\xi)x_1(\xi)\mathrm{d}\xi ,a \right)w\\
0
\end{bmatrix},\\
& l(x,a):=l_0(x_0,a), \quad g(x):=g_0(x_0) \quad \quad \quad  \quad \quad \quad \quad \quad \quad \quad \quad \forall x=(x_0,x_1) \in H, \ a \in \Lambda_0, \ w \in \mathbb{R}^q,
\end{align*}
the state equation \eqref{eq:SDDE} can be equivalently reformulated as the evolution equation \eqref{state} in the Hilbert space $H$, see \cite[Section 3]{deFeo-Federico-Swiech}. Moreover, as proved in \cite[Sections 3, 4]{deFeo-Federico-Swiech}, $B:=(A^{-1})^*A^{-1}$ satisfies the weak $B$-condition (Assumption \ref{assumptionAw}) with $C_0=0$ and the following crucial inequalities hold
\begin{align}
    &|x_0|\leq \|x\|_{-1} & \forall x=(x_0,x_1) \in H,\label{eq:x_0_leq_x_-1}\\
    & \left |\int_{-d}^0 \eta_i(\xi) x_1(\xi) d \xi \right| \leq C \|x\|_{-1} & \forall i=1,2, x=(x_0,x_1) \in H. \label{eq:int_leq_x_-1}
\end{align}
 \subsubsection{$C^{1,1}$ Regularity of the Value Function}
In the following, we show that the hypotheses made for the general theory are satisfied in the present case, assuming appropriate assumptions on the coefficients of the control problem. 
\begin{enumerate}
\item Let $b_0 \colon \mathbb{R}^n \times \mathbb{R}^h \times \Lambda_0 \to \mathbb{R}^n$, $\sigma_0 \colon \mathbb{R}^n \times \mathbb{R}^h \times \Lambda_0 \to \mathbb R^{n\times q}$ be continuous and be Lipschitz in the first two variables, uniformly in $a \in  \Lambda_0,$ and let there be a constant $C>0$ such that  $|b_0(y,z,a)|,|\sigma_0(y,z,a)|\leq C(1+|y|+|z|+|a|)$ for every  $  y \in \mathbb R^n,z \in \mathbb R^h,a \in \Lambda_0.$
Then, by \eqref{eq:x_0_leq_x_-1} and \eqref{eq:int_leq_x_-1},  Assumptions \ref{bsigmalipschitzfirstvariablew} and \ref{bsigmalipschitzw} are satisfied. 
\item Let $g_0 \colon \mathbb{R}^n  \to \mathbb{R}$ be Lipschitz and $l_0 \colon \mathbb{R}^n \times \Lambda_0 \to \mathbb{R}$ be continuous and be Lipschitz in the first variable, uniformly in $a \in  \Lambda_0,$ and let there be a constant $C>0$ such that  $|l_0(y,a)|\leq C(1+|{y}|+|a|^2)$ for every  $y \in \mathbb R^n,a \in \Lambda_0.$ If $\Lambda_0$ is unbounded we also assume $l_0,g_0$ bounded from below. Then Assumption    \ref{lglipschitzfirstvariablew} is satisfied. 
    \item For simplicity of notation we assume here that $n=h=q=1$ but the general case $n,h,q\in \mathbb{N}$ can be treated similarly. Let $b_0(\cdot,\cdot,a) \in C^{1,1}(\mathbb R \times \mathbb R;\mathbb R), \sigma_0(\cdot,\cdot,a) \in C^{1,1}(\mathbb R \times \mathbb R;\mathbb R)$ with the Lipschitz constants of $ \frac{\partial b_0}{\partial y}(\cdot, \cdot,a), \frac{\partial b_0}{\partial z}(\cdot, \cdot,a),\frac{\partial \sigma_0}{\partial y}(\cdot, \cdot,a), \frac{\partial \sigma_0}{\partial z}(\cdot, \cdot,a)$ being uniform in $a \in \Lambda$. Then Assumption  \ref{bsigmafrechetfirstvariablew} holds. We only show this for $b$ as similar techniques apply to $\sigma$. Notice that $b(x,a)$ is Fr\'echet differentiable in $x$ with Fr\'echet derivative $D_x b \colon H \times \Lambda_0 \rightarrow  L(H)$ given by
    \begin{align*}
        D_x b(x,a)=\begin{bmatrix}
            \frac{\partial b_0}{\partial y} \left ( x_0,\int_{-d}^0 \eta_1(\xi)x_1(\xi)\mathrm{d}\xi ,a \right) & \frac{\partial b_0}{\partial z} \left ( x_0,\int_{-d}^0 \eta_1(\xi)x_1(\xi)\mathrm{d}\xi ,a \right) \int_{-d}^0 \eta_1(\xi) \cdot \mathrm{d}\xi \\
        0 & 0
        \end{bmatrix}.
    \end{align*}
Using \eqref{eq:x_0_leq_x_-1}, \eqref{eq:int_leq_x_-1}, it holds
\begin{align*}
    &\| (D_xb(x,a) - D_xb(x^{\prime},a))(x-x^{\prime}) \|_{-1} \leq C \| (D_xb(x,a) - D_xb(x^{\prime},a))(x-x^{\prime}) \|_H\\
    &=\Bigg |\left ( \frac{\partial b_0}{\partial y} \left ( x_0,\int_{-d}^0 \eta_1(\xi)x_1(\xi)\mathrm{d}\xi ,a \right) - \frac{\partial b_0}{\partial y} \left ( x_0',\int_{-d}^0 \eta_1(\xi)x_1'(\xi)\mathrm{d}\xi ,a \right) \right)(x_0-x_0')\\
    & \quad \quad + \left ( \frac{\partial b_0}{\partial z} \left ( x_0,\int_{-d}^0 \eta_1(\xi)x_1(\xi)\mathrm{d}\xi ,a \right) -  \frac{\partial b_0}{\partial z} \left ( x_0',\int_{-d}^0 \eta_1(\xi)x_1'(\xi)\mathrm{d}\xi ,a \right)\right) \int_{-d}^0 \eta_1(\xi) (x_1-x_1')(\xi) \mathrm{d}\xi \Bigg |\\
    & \leq C \left ( |x_0-x_0'| + \Bigg |\int_{-d}^0 \eta_1(\xi)(x_1-x_1')(\xi)\mathrm{d}\xi \Bigg | \right  )^2 \leq C  \|x-x' \|^2_{-1}.
\end{align*}
    \item Let $\sigma_0$ be independent of $a$ and $b_0 \in C^{1,1}(\mathbb R^n \times\mathbb R^h \times \mathbb R^n;\mathbb R^n),$ $ \sigma_0 \in C^{1,1}(\mathbb R^n \times \mathbb R^h;\mathbb R^{n \times q})$.  Then Assumption \ref{bsigmafrechetw} is proved using similar techniques as above.
    \item  Let  $g_0(\cdot)$ be semiconcave and $l_0(\cdot,a)$ be semiconcave with semiconcavity constant uniform in $a \in \Lambda_0$. Then Assumption \ref{lgsemiconcavew} is satisfied. Indeed, by \eqref{eq:x_0_leq_x_-1} we have 
    \begin{align*}
        \lambda l(x,a)+ (1-\lambda) l(x',a)-l(\lambda x+ (1-\lambda)x',a )& =      \lambda l_0(x_0,a)+ (1-\lambda) l_0(x_0',a)-l_0(\lambda x_0+ (1-\lambda)x_0',a )\\& \leq C \lambda (1-\lambda)|x_0-x_0'|^2 \leq C \lambda (1-\lambda)|x-x'|^2_{-1}
    \end{align*}
    for every $\lambda \in [0,1], $ $x=(x_0,x_1), x'=(x_0',x_1') \in H, a \in \Lambda_0,$ so $l_0(\cdot,a)$ is semiconcave. The same computation applies to $g$.
  \item Let $g_0(x_0)$ be semiconvex and let there exist  $C,\nu \geq 0$ such that the map $(x_0,a)\mapsto l_0(x_0,a) + C|x_0|^2 - \nu |a|^2$
		is convex. Then Assumption \ref{lgsemiconvexw} is satisfied by a similar argument as in point 5 above. As before, using that the function
	 $\mathbb R^n \times\Lambda_0 \ni (x_0,a)\mapsto l_0(x_0,a) + C|x_0|^2 - \nu |a|^2$ is convex and \eqref{eq:x_0_leq_x_-1}, we have
   \begin{align*}
         &\lambda l_0(x_0,a)+ (1-\lambda) l_0(x_0',a')-l_0(\lambda x_0+ (1-\lambda)x_0',\lambda a+ (1-\lambda)a' ) 
         \\& \geq - C \lambda (1-\lambda)|x_0-x_0'|^2 +  \nu \lambda (1-\lambda)|a-a'|^2
         \geq - C \lambda (1-\lambda)|x-x'|^2_{-1} +  \nu \lambda (1-\lambda)|a-a'|^2
    \end{align*}
     for every $\lambda \in [0,1],$ $x=(x_0,x_1), x'=(x_0',x_1') \in H,$  $a,a' \in \Lambda_0. $ This means that  the function $H \times\Lambda_0 \ni(x,a)\mapsto l(x,a) + C|x|^2_{-1} - \nu |a|^2$ is convex.
    \item  Let $b_0,\sigma_0$  be linear  and $l_0,g_0$ be convex. Then Assumption \ref{blinearlgconvexw} is  satisfied.
\end{enumerate}
\begin{remark}\label{rem:case3_delay}
We also point out that in the infinite horizon case in \cite[Example 4.2]{deFeoSwiech}, convexity of $V$ is obtained by means of comparison theorems for SDDEs in the spirit of Section \ref{section:case3}. In the finite horizon case here similar techniques apply.
\end{remark}
Hence, using the theory developed in Section \ref{weakBcase}, we can obtain $V(t,\cdot) \in C^{1,1}(H_{-1})$ under various combinations of the above assumptions. We recall that in \cite{deFeo-Federico-Swiech}, only partial $C_{\rm loc}^{1,\alpha}(\mathbb R^n)$ regularity of $V$ with respect to $x_0$ was obtained, however the assumptions used there were very weak.
  \subsubsection{Optimal Synthesis}
Since by previous subsection we can obtain $V(t,\cdot) \in C^{1,1}(H_{-1})$, we now assume that Assumption \ref{hp:V_C_11w} holds.
Let $\sigma_0$ be independent of $a$ and let $\Lambda_0=\Lambda=\mathbb R^n$. Note that we are not exactly in the setting of Example \ref{exampleinfimumattained} as we do not have $ \Lambda_0= H$. Moreover, let $b_0 \colon \mathbb{R}^n \times \mathbb{R}^h \times \mathbb{R}^n \to \mathbb{R}^n$ be of the form $b_0(y,z,a)=b_0(y,z)-a$ for every $(y,z,a) \in \mathbb R^n \times \mathbb R^h \times \mathbb R^n$ and  $l_0(x_0,a)=l_0^1(x_0)+l_0^2(a)$. 
\begin{remark}
We remark that this kind of structure of the coefficients is frequently employed in the economic literature. Examples in the case of delays are optimal advertising problems (see e.g. \cite{gozzi_marinelli_2004}, \cite{deFeo-Federico-Swiech}) or optimal control of endogenous growth models with time-to-build (e.g. see e.g. \cite{bambi_fabbri_gozzi (2012)}, \cite{goldys_2} for the deterministic case).
\end{remark}
We now have 
 \begin{equation*}
        \mathcal{F}(x,p,a) =  p_0 \cdot \left [ b_0 \left ( x_0,\int_{-d}^0 \eta_1(\xi)x_1(\xi)\mathrm{d}\xi \right)+x_0 \right]  -  p_0  \cdot a  + l_0^1(x_0) + l_0^2(a)  \quad \forall x=(x_0,x_1), p=(p_0,p_1) \in  H, a \in \mathbb R^n.
    \end{equation*} 
    Let $\mathcal H$ satisfy Assumption \ref{coercivity} (e.g. recall Remark \ref{rem:H_attains_inf_bounded_set_H}). Let $l_0^2 \in C^{1,1}(\mathbb R^n)$ and assume that $l_0^2(a)-\nu |a|^2$ is convex for some $\nu>0$. Thus we can proceed in a similar way to Example \ref{exampleinfimumattained} to show that $Dl_0^2 \colon \mathbb R^n \to \mathbb R^n$ is invertible with Lipschitz derivative $(Dl_0^2)^{-1}\colon \mathbb R^n \to \mathbb R^n$. This means that the infimum in the Hamiltonian is attained at $a^*=(Dl_0^2)^{-1}(p_0)$ so that Assumption \ref{infimumattainedw} holds  and we can construct optimal feedback controls by means of Theorem \ref{th:optsynthw}.
 We emphasize that here only $D_{x_0}V$ is used to construct optimal feedback controls.
\begin{remark}
    We point out that here we get uniqueness of solutions of the closed loop equation while in \cite{deFeoSwiech} only existence of weak solutions was obtained using Girsanov's theorem.
\end{remark}

\appendix
\section{A Technical Result}
In this section we prove a general result regarding uniform continuity of the derivative of a function $V \colon [0,T]\times H \to \mathbb R$ on bounded sets.
\begin{lemma}\label{lemma:DV_locally_uniform_continuous} Let $V \colon [0,T]\times H \to \mathbb R$ be uniformly continuous on $[0,\tau]\times B_R$ for every $0<\tau<T$ and $R>0$. Moreover, let for every $0\leq t \leq T$, $V(t,\cdot)\in C^{1,\alpha}(H)$ for some $0<\alpha\leq 1$ and let there exist $C\geq 0$ such that $\|DV(t,x)\|_H\leq C, \|DV(t,x)-DV(t,x')\|_H\leq C\|x-x'\|^\alpha_H$ for all $0\leq t \leq T,x,x'\in H$. Then $DV$ is uniformly continuous on $[0,\tau]\times B_R$ for every $0<\tau<T$ and $R>0$.
\end{lemma}
\begin{proof}
Fix $R>0$ and assume, for the sake of contradiction, that there exist $\varepsilon>0$, $(t_n,x_n)_{n\in{\mathbb{N}}} ,(s_n,y_n)_{n\in{\mathbb{N}}} \subset [0,\tau] \times B_R$ such that $|t_n-s_n| \to 0, \|x_n-y_n\|_H \to 0$ and $\|DV(t_n,x_n)- DV(s_n,y_n)\|_H \geq \varepsilon.$ We remark that in the following, the constants $C$ may depend on $R$. Define 
\begin{equation*}
    V_n(t,x):=V(t,x)-V(s_n,y_n)- \langle DV(s_n,y_n),x-y_n\rangle_H.
\end{equation*}
Note that $V_n$ is uniformly continuous on $[0,\tau] \times B_R$ with a modulus of continuity $\rho$ independent of $n$ and $V_n(t, \cdot) \in C^{1,\alpha}$ for every $t$ with Holder's continuity constant of $DV_n$ independent of $n$ and $t \leq T$. Defining $DV_n(t,x)=DV(t,x)-DV(s_n,y_n),$  we have
\begin{equation}\label{eq:propVn}
    V_n(s_n,y_n)=0, \quad DV_n(s_n,y_n)=0, \quad  \|DV_n(t_n,x_n)\|_H=\|DV(t_n,x_n)- DV(s_n,y_n)\|_H \geq \varepsilon.
\end{equation}
Hence, for every $y \in H$ such that
\begin{equation}\label{eq:norm_y-y_n}
    \| y-y_n\|_H \leq (\rho(|t_n-s_n|))^{\frac 1 {1+\alpha}}+2\|x_n-y_n \|_H,
\end{equation}
by the Mean Value Theorem (denoting $y_{\theta_n}:=\theta_n y + (1-\theta_n)y_n$ for some $\theta_n \in [0,1]$) we have
\begin{equation*}
\begin{split}
    |V_n(s_n,y)|&=|V_n(s_n,y)-V_n(s_n,y_n)|=| \langle DV_n(s_n,y_{\theta_n}),y-y_n \rangle_H| =| \langle DV_n(s_n,y_{\theta_n})- DV_n(s_n,y_n),y-y_n \rangle_H| \\
    &\leq C \|y-y_n\|_H^{1+\alpha} \leq C \left[\rho(|t_n-s_n|)+\|x_n-y_n \|_H^{1+\alpha} \right],
\end{split}
\end{equation*}
where we have also used the fact that $V_n(t, \cdot) \in C^{1,\alpha}$ for every $t \leq T$ with Lipschitz constant of $DV_n$ independent of $n$ and $t \leq T$. It follows 
\begin{equation}\label{eq_norm_v_n_t_n_y}
    |V_n(t_n,y)|\leq |V_n(s_n,y)| + \rho(|t_n-s_n|)\leq C \left[\rho(|t_n-s_n|)+\|x_n-y_n \|_H^{1+\alpha} \right].
\end{equation}
Defining $y \in H$ by 
\begin{equation}\label{eq:def_y_proof_unif_regularityDV}
y:=x_n+\frac{DV_n(t_n,x_n)}{\|DV_n(t_n,x_n) \|_H} \left [(\rho(|t_n-s_n|))^{\frac 1 {1+\alpha}}+\|x_n-y_n \|_H  \right ],
\end{equation}
we have \eqref{eq:norm_y-y_n}, so that  \eqref{eq_norm_v_n_t_n_y} holds. Hence, we have a first inequality for $V_n(t_n,y)$: 
\begin{equation}\label{inequalityVone}
C \left[\rho(|t_n-s_n|)+\|x_n-y_n \|_H^{1+\alpha} \right] \geq  V_n(t_n,y).
\end{equation}
Note also that $x_n$ satisfies \eqref{eq:norm_y-y_n}, so that \eqref{eq_norm_v_n_t_n_y} holds for $V_n(t_n,x_n)$. Then, since $V_n(t, \cdot) \in C^{1,\alpha}$ for every $t \leq T$ with Lipschitz constant of $DV_n$ independent of $n$ and $t \leq T$, we have
\begin{equation}\label{inequalityVtwo}
\begin{split}
    V_n(t_n,y) & \geq V_n(t_n,x_n) + \left \langle DV_n(t_n,x_n), y-x_n \right \rangle_H  - C \|y-x_n\|_H^{1+\alpha}\\
    & =V_n(t_n,x_n) +\|DV_n(t_n,x_n) \|_H  \left [ (\rho(|t_n-s_n|))^{\frac 1 {1+\alpha}}+\|x_n-y_n \|_H  \right ]- C \|y-y_n\|_H^{1+\alpha}\\
    & \geq V_n(s_n,y_n) -  C \left[\rho(|t_n-s_n|)+\|x_n-y_n \|_H^{1+\alpha} \right] + \varepsilon \left [ (\rho(|t_n-s_n|))^{\frac 1 {1+\alpha}}+2\|x_n-y_n \|_H  \right ]\\
    & = -  C \left[\rho(|t_n-s_n|)+\|x_n-y_n \|_H^{1+\alpha} \right] + \varepsilon \left [ (\rho(|t_n-s_n|))^{\frac 1 {1+\alpha}}+2\|x_n-y_n \|_H  \right ]
\end{split}
\end{equation}
where in the second line we have used \eqref{eq:def_y_proof_unif_regularityDV} and in the last two lines we have used  \eqref{eq:propVn} and \eqref{eq:norm_y-y_n}. Combining the two inequalities \eqref{inequalityVone} and \eqref{inequalityVtwo} for $V_n(t_n,y)$, we obtain
\begin{equation*}
    C \left[\rho(|t_n-s_n|)+\|x_n-y_n \|_H^{1+\alpha} \right] \geq \varepsilon \left [ (\rho(|t_n-s_n|))^{\frac 1 {1+\alpha}}+2\|x_n-y_n \|_H  \right ],
\end{equation*}
which is a contradiction for large $n$.
\end{proof}

\paragraph{\textbf{Acknowledgments:}} 
Filippo de Feo acknowledges support from DFG CRC/TRR 388 "Rough Analysis, Stochastic Dynamics and Related Fields", Project B05, by INdAM (Instituto Nazionale di Alta Matematica F. Severi) - GNAMPA (Gruppo Nazionale per l'Analisi Matematica, la Probabilità e le loro Applicazioni), and by the Italian Ministry of University and Research (MUR), in the framework of PRIN projects 2017FKHBA8 001 (The Time-Space Evolution of Economic Activities: Mathematical Models and Empirical Applications) and 20223PNJ8K (Impact of the Human Activities on the Environment and Economic Decision Making in a Heterogeneous Setting: Mathematical Models and Policy Implications). Lukas Wessels was partially supported by a fellowship of the German Academic Exchange Service (DAAD).

\end{document}